\documentclass{article}
\usepackage{graphicx} 

\usepackage{amsmath, amsthm,    amsfonts, amssymb}
\usepackage{bm}
\usepackage{xcolor}

\usepackage[a4paper, total={19cm, 27cm}]{geometry}

\usepackage{subcaption}
\usepackage{aligned-overset}
\usepackage{hyperref}

\newcommand{\energy}{\mathcal{E}}
\newcommand{\pf}{\varphi}

\newcommand{\strain}[1][u]{{\bm \varepsilon}(\bm #1)}

\newcommand{\Div}{\nabla\cdot}

\newcommand{\calH}{\mathcal{H}}
\DeclareMathOperator*{\argmin}{arg\,min}
\DeclareMathOperator{\tr}{\mathrm{tr}}
\newcommand{\vecSpL}{\bm{L^2(\Omega)}}

\newcommand{\vecSpHTr}{\bm{H_0^1(\Omega)}}

\newcommand{\tensSpL}{\bm{\mathcal{L}^2(\Omega)}}

\usepackage{tikz}
\usepackage{pgfplots,pgfplotstable}
\pgfplotsset{compat=newest}

\newtheorem{lemma}{Lemma}

\newtheorem{theorem}{Theorem}

\theoremstyle{definition}

\theoremstyle{remark}
\newtheorem{remark}{Remark}

\definecolor{blau}{RGB}{0 144 188}
\definecolor{newblue1}{RGB}{0 144 188}
\definecolor{newblue2}{RGB}{197 216 227}
\definecolor{newgreen1}{RGB}{0 144 118}
\definecolor{newgreen2}{RGB}{197 222 215}
\definecolor{neworange1}{RGB}{255 137 0}
\definecolor{neworange2}{RGB}{255 205 105}
\definecolor{newred1}{RGB}{254 54 41}
\definecolor{newred2}{RGB}{141 26  18}
\definecolor{newpurple1}{RGB}{196 19 252}
\definecolor{newpurple2}{RGB}{93 14 117}

\usepackage{authblk}

\providecommand{\keywords}[1]{\quad {\small \textbf{\textit{Keywords:}} #1}}

\usepackage{chngcntr}

\title{On the convergence and efficiency of splitting schemes for the Cahn-Hilliard-Biot model}

\author[1]{Cedric Riethmüller\thanks{Corresponding author: cedric.riethmueller@ians.uni-stuttgart.de}}

\author[2]{Erlend Storvik}

\affil[1]{{\small Institute of Applied Analysis and Numerical Simulation, University of Stuttgart, Germany}}

\affil[2]{{\small Department of Computer science, Electrical engineering and Mathematical sciences, Western Norway University of Applied sciences, Norway}}

\date{}

\begin{document}

\maketitle

\begin{abstract}
In this paper, we present a novel solution strategy for the Cahn-Hilliard-Biot model, a three-way coupled system that features the interplay of solid phase separation, fluid dynamics, and elastic deformations in porous media. It is a phase-field model that combines the Cahn-Hilliard regularized interface equation and Biot's equations of poroelasticity. Solving the system poses significant challenges due to its coupled, nonlinear, and non-convex nature. The main goal of this work is to provide a consistent and efficient solution strategy. With this in mind, we introduce a semi-implicit time discretization such that the resulting discrete system is equivalent to a convex minimization problem. Then, using abstract theory for convex problems, we prove the convergence of an alternating minimization method to the time-discrete system. The solution strategy is relatively flexible in terms of spatial discretization, although we require standard inverse inequalities for the guaranteed convergence of the alternating minimization method. Finally, we perform some numerical experiments that show the promise of the proposed solution strategy, both in terms of efficiency and robustness.
\end{abstract}


\section{Introduction}

The Cahn-Hilliard-Biot model combines the quasi-static Biot equations for flow in deformable porous media with the Cahn-Hilliard equation for modeling phase-field evolution to get a system that captures fluid flow in deformable porous media with evolving solid phases. There are several relevant applications ranging from the life-sciences with tumor growth, to material science with wood growth and biogrout. The equations were first proposed in \cite{Storvik2022} and since then several contributions have been made to enhance the understanding of the model. The well-posedness of the system has been addressed in \cite{Fritz2024, Garcke2025, Riethmueller2025, Abels2025}. Furthermore, the sharp-interface limit of the system was studied using asymptotic expansions in \cite{Storvik2025}. 

To the authors' knowledge, there exist two works that consider numerical solution strategies to the Cahn-Hilliard-Biot model. In \cite{Brunk2025}, a structure-preserving discretization is presented for a regularized form of the model that incorporates viscoelastic effects. The authors provide theoretical proofs that the discretization is gradient stable and provides global balance in both phase field and volumetric fluid content. Furthermore, a proof that there exists a unique solution to the discrete system is provided in addition to a priori bounds, numerical experiments indeed also show the expected convergence properties of the discretization method.
In \cite{Storvik2024}, the focus is on the solution strategies of the discrete system, which is discussed, mostly by comparative numerical experiments. Splitting strategies to iteratively decouple the system of equations during the solving process were shown to be beneficial in comparison to standard Newton methods, in particular due to the efficiency of the assembly and linear solvers on the subsystems, contrary to the full system. In this paper, we aim to provide a time discretization that is suitable to solve with tailored iterative linearization and decoupling strategies. 

One particularly useful and important trait of the Cahn-Hilliard-Biot equations is that they have a generalized gradient flow structure \cite{Storvik2022}. This means that the state of the system evolves subject to the gradient of the free energy of the system. In \cite{Garcke2025, Riethmueller2025, Abels2025} this was partially used to establish well-posedness for the system, and in \cite{Storvik2023} the same structure was used to find a tailored time discretization to the Cahn-Larch\'e equations, based on the classical convex-concave split of the Cahn-Hilliard equations \cite{Eyre1998}. In the present paper, the goal is to perform a similar discretization to the one in \cite{Storvik2023}, but for the Cahn-Hilliard-Biot equations, leading to a system of discrete equations that is equivalent to a convex minimization problem. 

Utilizing the equivalence between the discrete equations and a convex minimization problem, we aim to apply the abstract theory from \cite{Both2022} to devise an alternating minimization solution strategy to the problem. This strategy will be theoretically proven to converge, and we show that it leads to sequentially, iteratively solving a discrete Cahn-Hilliard-like system and a Biot-like system of equations. From there, several options for numerically approximating the solutions to the discrete subproblems remain. Where the Cahn-Hilliard subsystem seems best fitted to be solved with a standard Newton method at this point, we can either solve the Biot subsystem monolithically, as it at each iteration of the alternating minimization step will be equivalent to a discrete heterogeneous Biot system of equations, or we could use another sequential decoupling strategy, like the fixed-stress splitting scheme \cite{Settari1998, Mikelic2013, Storvik2019}, or the undrained split \cite{Kim2011}. In this paper, we aim to test both options numerically and compare their performance. 

The outline of the paper is as follows. We start by introducing the mathematical formulation of the Cahn-Hilliard-Biot model in Section~\ref{sec:model} along with assumptions on the model parameters. In Section~\ref{sec:solution strategy} we focus on the numerical solution where the considerations are split into the case of state-independent material parameters (cf. Section~\ref{sec:state-independent}) and the state-dependent one (cf. Section~\ref{sec:state-dependent}). Based on a novel semi-implicit time discretization (cf. Section~\ref{sec:alter_min} and Section~\ref{sec:alter_min_sd} for the respective case), we establish convergence proofs for the alternating minimization method. In Section~\ref{sec:experiments} we explore solution algorithms applied to the proposed discrete system, comparing monolithic and splitting strategies, and assess their computational efficiency and robustness in numerical experiments. We end with concluding remarks in Section~{\ref{sec:conclusions}}.

\section{Model formulation} \label{sec:model}

Let $\Omega \subset \mathbb{R}^d$, $d \in \{2,3\}$, be a domain with Lipschitz continuous boundary and $T > 0$ be the final time. We then consider the problem: Find the quintuple $(\pf, \mu, \bm u, \theta, p)$, with phase-field variable $\pf$, chemical potential $\mu$, displacement $\bm u$, volumetric fluid content $\theta$ and pore pressure $p$, such that
\begin{subequations}
\label{eq:model}
\begin{align}
\partial_t \pf - \Div (m(\pf) \nabla \mu) &= R, \label{eq:ch1}\\
\mu + \gamma\left(\ell\Delta\pf-\frac{1}{\ell}\Psi'(\pf)\right) - \delta_\pf\mathcal{E}_\mathrm{e}(\pf, \strain) - \delta_\pf \mathcal{E}_\mathrm{f}(\pf, \strain, \theta) &= 0, \label{eq:ch2}\\
 -\Div\big(\mathbb{C}(\pf)\left(\bm\varepsilon\!\left(\bm u\right)-\mathcal{T}(\pf)\right)-\alpha(\pf) M(\pf)\left(\theta - \alpha(\pf) \Div \bm u\right) \bm I\big) &= {\bm f}, \label{eq:elasticity}\\
\partial_t \theta - \Div \left(\kappa(\pf) \nabla p\right) &= S_\mathrm{f}, \label{eq:flow}\\
p - M(\pf)\left(\theta - \alpha(\pf) \Div \bm u\right) &= 0 \label{eq:pressure}
\end{align}
\end{subequations}
in $\Omega \times [0,T]$, where
\begin{align*}
\delta_\pf\mathcal{E}_\mathrm{e}(\pf, \strain) = - \mathcal{T}'\!(\pf)\!:\!\mathbb{C}(\pf)\big(\bm\varepsilon\!\left(\bm u\right)-\mathcal{T}(\pf)\big)+\frac{1}{2} \big(\strain - \mathcal{T}(\pf)\big)\!:\!\mathbb{C}'\!(\pf)\big(\strain - \mathcal{T}(\pf)\big)
\end{align*}
and
\begin{align*}
\delta_\pf\mathcal{E}_{\mathrm{f}}(\pf, \strain, \theta) = \frac{M'(\pf)}{2} \left(\theta - \alpha(\pf) \Div \bm u\right)^2 - \alpha'(\pf) M(\pf)\left(\theta - \alpha(\pf) \Div \bm u\right) \Div \bm u,
\end{align*}
accompanied by appropriate initial and boundary conditions. Here, the state-dependent (material) parameters are the chemical mobility $m(\pf)$, the permeability $\kappa(\pf)$, the Biot modulus $M(\pf)$ (with integrated fluid compressibility), the Biot-Willis coupling coefficient $\alpha(\pf)$, the stiffness tensor $\mathbb{C}(\pf)$ and the eigenstrain $\mathcal{T}(\pf)$. Moreover, $\gamma > 0$ denotes the surface tension, and $\ell > 0$ is a small regularization parameter that is associated with the width of the diffuse interface. $\Psi(\pf)$ is a double-well potential that penalizes deviations from the pure phases ($\pf = -1$ and $\pf = 1$) 
and $\strain := \frac{1}{2} \left(\nabla \bm u + \nabla \bm u^T\right)$ is the linearized strain tensor (with $\Div \bm u = \tr{\strain}$). Concerning the right-hand side, $R$ is a reaction term, $\bm f$ accounts for body forces, and $S_\mathrm{f}$ is a source term.

We note that the Cahn-Hilliard-Biot model~\eqref{eq:model} obeys a generalized gradient flow structure, cf.~\cite{Storvik2022, Both2019}, with free energy
\begin{align*}
 \energy(\pf, \strain, \theta) = \energy_\mathrm{i}(\pf) + \energy_\mathrm{e}(\pf, \strain) + \energy_\mathrm{f}(\pf, \strain, \theta)
\end{align*}
that is composed by the interfacial energy
\begin{align*}
 \energy_\mathrm{i}(\pf) = \gamma \int_\Omega \frac{1}{\ell} \Psi(\pf) + \frac{\ell}{2} |\nabla \pf|^2 \;dx,
\end{align*}
the elastic energy
\begin{align*}
 \energy_\mathrm{e}(\pf, \strain) = \int_\Omega \frac{1}{2} \big(\bm\varepsilon(\bm u) - \mathcal{T}(\pf)\big) \!:\!\mathbb{C}(\pf) \big(\bm\varepsilon(\bm u) - \mathcal{T}(\pf)\big) \;dx
\end{align*}
and the hydraulic energy
\begin{align*}
 \energy_\mathrm{f}(\pf, \strain, \theta) = \int_\Omega \frac{M(\pf)}{2} (\theta - \alpha(\pf) \Div \bm u)^2 \;dx.
\end{align*}
Moreover, the equations consist of a Cahn-Hilliard-like subsystem \eqref{eq:ch1}-\eqref{eq:ch2}, and a Biot-like subsystem combining linearized elasticity (with momentum balance) \eqref{eq:elasticity}, and single-phase Darcy flow (with mass balance) \eqref{eq:flow}-\eqref{eq:pressure}.

\subsection{Notation}
In this paper, we apply standard notation from functional analysis. Let $L^2(\Omega)$ be the space of square-integrable functions, and $H^1(\Omega)$ the Sobolev space of square-integrable functions with square-integrable derivatives. Furthermore, $\left(H_0^1(\Omega)\right)^d$ denotes the space of vector-valued $H^1(\Omega)$ functions with zero trace at the boundary. For ease of notation, we define the spaces for vector functions of dimension $d$ as $\vecSpL := \left(L^2(\Omega)\right)^d$ and $\vecSpHTr := \left(H_0^1(\Omega)\right)^d$, and additionally the space of tensor-valued $L^2(\Omega)$ functions $\tensSpL := \left(L^2(\Omega)\right)^{d \times d}$. The spaces are equipped with norms $\| \cdot \|_X$, and inner products $( \cdot,\cdot )_X$. If the subscripts are omitted we refer to the respective $L^2(\Omega)$-versions. Moreover, we use the theory of Bochner spaces (e.g.~\cite{Boyer2013}, §5; \cite{TRoubi2013}, §1.5, §7.1). For a given Banach space $X$, we define
\begin{align*}
 L^2([0,T];X) := \left\{f\!:[0,T] \rightarrow X \Big| \ f \textrm{ Bochner measurable, } \int_0^T \|f(t)\|_X^2\; dt < \infty\right\}
\end{align*}
and
\begin{align*}
 H^1([0,T];X) := \left\{f \in L^2([0,T];X) \Big| \ \partial_t f \in L^2([0,T];X)\right\} \textrm{ (cf.~\cite{Evans2010}, §5.9.2).}
\end{align*}
In addition, we will frequently make use of classical inequalities such as the Cauchy-Schwarz and Young's inequalities (see e.g.~\cite{Hardy1988}). For the Poincar\'e-Wirtinger inequality we refer to Proposition III.2.39 of \cite{Boyer2013}.

\subsection{Assumptions on model parameters}
Throughout this work we require the following assumptions to hold:

\begin{itemize}
\item[(A1)]
The double-well potential $\Psi(\pf)$ is non-negative, i.e., $\Psi(\pf) \geq 0$ for all $\pf \in \mathbb{R}$.
Moreover, we assume that the double-well potential allows for a convex-concave split in the form 
\begin{align*}
 \Psi(\pf) = \Psi_c(\pf) - \Psi_e(\pf)
\end{align*}
with convex functions $\Psi_c(\pf)$ and $\Psi_e(\pf)$. $\Psi_c(\pf)$ and $\Psi_e(\pf)$ are differentiable, and $\Psi_c'(\pf)$ is Lipschitz continuous with Lipschitz constant $L_{\Psi_c'}$.

\item[(A2)]
Mobility $m(\pf) = m$ and permeability $\kappa(\pf) = \kappa$ are state-independent and given by positive constants.
The Biot modulus $M$ (with integrated fluid compressibility) and Biot-Willis coupling coefficient $\alpha$ as functions of $\pf$ are Lipschitz continuous, differentiable, and uniformly bounded such that there exist suitable constants $c_M,\ C_M,\ c_\alpha,\ C_\alpha$ satisfying for all $\pf \in \mathbb{R}$
\begin{alignat*}{8}
 0 &<& \hspace*{0.1cm} c_M &\leq& \hspace*{0.1cm} M(\pf) &\leq& \hspace*{0.1cm} C_M &<& \hspace*{0.1cm} \infty, \\ 
 0 &<& c_\alpha &\leq& \alpha(\pf) &\leq& C_\alpha &\leq& 1.
\end{alignat*}

\item[(A3)]
The elastic stiffness tensor $\mathbb{C}(\pf)$ is Lipschitz continuous and differentiable. We further assume that there exist positive constants $c_{\mathbb{C}}$ and $C_{\mathbb{C}}$ such that
\begin{align*}
 c_{\mathbb{C}}|\bm \varepsilon|^2 \leq \bm \varepsilon\!:\!\mathbb{C}(\pf) \bm \varepsilon \leq C_{\mathbb{C}}|\bm \varepsilon|^2
\end{align*}
holds for all $\pf \in \mathbb{R}$ and $\bm \varepsilon \in \mathbb{R}^{d\times d}_\mathrm{sym}$ and that the symmetry conditions on $\mathbb{C}$ known from linear elasticity, cf., e.g.,~\cite{Coussy2004}, are fulfilled.

\item[(A4)]
The eigenstrain $\mathcal{T}(\pf)$ is symmetric, Lipschitz continuous, differentiable, and bounded by the $L^2(\Omega)$ norm of $\pf$ such that for all $\pf \in L^2(\Omega)$
\begin{align*}
 c_{\mathcal{T}} \|\pf\|_{L^2(\Omega)} \leq \|\mathcal{T}(\pf)\|_{\tensSpL} \leq C_{\mathcal{T}} \|\pf\|_{L^2(\Omega)}
\end{align*}
for positive constants $c_\mathcal{T}$ and $C_\mathcal{T}$.

\item[(A5)]
The source terms are (for simplicity) assumed to be autonomous and satisfy $R, S_f \in L^2(\Omega)$ and $\bm f \in \vecSpL$ with uniform bounds.
\end{itemize}

\begin{remark}[Assumption (A1)] \label{rem:double-well}
For the classical choice of the double-well potential
\begin{align*}
 \tilde\Psi(\pf) = \left(1 - \pf^2\right)^2
\end{align*}
with the convex-concave split
\begin{align*}
 \tilde\Psi(\pf) = \left[\pf^4 + 1\right] - \left[2\pf^2\right] = \tilde\Psi_c(\pf) - \tilde\Psi_e(\pf)
\end{align*}
assumption (A1) is not fulfilled as $\tilde\Psi_c'(\pf)$ is not Lipschitz continuous. In order to do so, we introduce a modified double-well potential that is piecewisely defined. For constant $\beta > 1$, we have
\begin{align*}
 \Psi(\pf) =
 \begin{cases}
  \tilde\Psi(\pf) &\text{ if } \pf \in \left(-\beta, \beta\right) \\
  2\left(\beta^2 - 1\right) \pf^2 - \left(\beta^4 - 1\right) &\text{ else}
 \end{cases}
\end{align*}
and introduce the split $\Psi(\pf) = \Psi_c(\pf) - \Psi_e(\pf)$ with
\begin{align*}
 \Psi_c(\pf) =
 \begin{cases}
  \tilde\Psi_c(\pf) &\text{ if } \pf \in \left(-\beta, \beta\right) \\
  2 \beta^2 \pf^2 - \left(\beta^4 - 1\right) &\text{ else}
 \end{cases}
\end{align*}
and 
\begin{align*}
\Psi_e(\pf) = \tilde\Psi_e(\pf) = 2\pf^2.
\end{align*}
Then, $\Psi_c'(\pf)$ is Lipschitz continuous with Lipschitz constant $L_{\Psi_c'} = 4 \beta^2$. Note that (for a suitable choice of $\beta$) this modification does not alter the solution as $\pf$ rarely takes values outside of $[-1,1]$.
\end{remark}

\begin{remark}[Assumption (A4)]
Note that assumption (A4) is satisfied for the linear ansatz of Vegard's law \cite{Garcke2021}, that is, $\mathcal{T}(\pf) = \bm{\hat \varepsilon} \pf$ with constant symmetric tensor $\bm{\hat \varepsilon}$, and thus also for the choice $\mathcal{T}(\pf) = \xi \pf \bm I$ \cite{Areias2016} that accounts for swelling effects with (positive) constant $\xi$.
\end{remark}

\section{Numerical solution strategy} \label{sec:solution strategy}

This section is devoted to deriving an efficient and robust numerical solution strategy for the Cahn-Hilliard-Biot model based on a tailor-made semi-implicit time discretization. The following considerations are split into the case of state-independent material parameters in Section~\ref{sec:state-independent} and the state-dependent one in Section~\ref{sec:state-dependent}.

\subsection{Variational form}

The continuous variational formulation of system \eqref{eq:model} considered here reads: \\
Find $\pf \in H^1\big([0,T];H^1(\Omega)\big)$, $\mu \in L^2\big([0,T];H^1(\Omega)\big)$, $\bm u \in L^2\big([0,T];\vecSpHTr\big)$, $\theta \in H^1\big([0,T];L^2(\Omega)\big)$ and $p \in L^2\big([0,T];H^1(\Omega)\big)$ such that
\begin{subequations}
\label{eq:model_var}
\begin{align}
\left(\partial_t \pf,\eta^\pf\right) + \left(m \nabla \mu, \nabla \eta^\pf\right) - \left(R,\eta^\pf\right) &= 0, \label{eq:ch1_var}\\
\left(\mu,\eta^\mu\right) - \gamma\left(\ell\left(\nabla\pf,\nabla\eta^\mu\right)+\frac{1}{\ell}\left(\Psi'(\pf),\eta^\mu\right)\right) - \left(\delta_\pf\mathcal{E}_\mathrm{e}\!\left(\pf, \strain\right),\eta^\mu\right) - \left(\delta_\pf \mathcal{E}_\mathrm{f}\!\left(\pf, \strain, \theta\right),\eta^\mu\right) &= 0, \label{eq:ch2_var}\\
\left(\mathbb{C}(\pf)\left(\bm\varepsilon\!\left(\bm u\right)-\mathcal{T}(\pf)\right),\bm\varepsilon\!\left(\bm \eta^{\bm u}\right)\right) - \left(\alpha(\pf) M(\pf)\left(\theta - \alpha(\pf) \Div \bm u\right), \Div \bm \eta^{\bm u}\right)-\left({\bm f},\bm \eta^{\bm u}\right) &= 0, \label{eq:elasticity_var}\\
\left(\partial_t \theta, \eta^\theta\right) + \left(\kappa \nabla p, \nabla \eta^\theta\right) - \left(S_\mathrm{f}, \eta^\theta\right) &= 0, \label{eq:flow_var}\\
\left(p, \eta^p\right) - \left(M(\pf)\left(\theta - \alpha(\pf) \Div \bm u\right), \eta^p\right) &= 0 \label{eq:pressure_var}
\end{align}
\end{subequations}
for all $\eta^\pf$, $\eta^\mu \in H^1(\Omega)$, $\bm \eta^{\bm u} \in \vecSpHTr$, $\eta^\theta \in H^1(\Omega)$ and $\eta^p \in L^2(\Omega)$, and a.e.~$t \in [0,T]$.

\subsection{Discrete function spaces and norms}

For the spatial discretization (indicated by subscript $h$) with mesh diameter $h$, let $\mathcal{V}_h^{\mathrm{ch}} \subseteq H^1(\Omega)$, $\mathcal{V}_h^{\bm u} \subseteq \vecSpHTr$ and $\mathcal{V}_h^{\mathrm{f}} \subseteq H^1(\Omega)$ be conforming finite element spaces, serving as the solution spaces for the Cahn-Hilliard, elasticity and flow subproblems respectively.

The subsequent analysis necessitates further specification of the spaces for the Cahn-Hilliard and the flow subproblem. Due to their structural similarity, we introduce the auxiliary variables $y \in \{m, \kappa\}$, $\mathrm{z} \in \{\mathrm{ch}, \mathrm{f}\}$ with feasible pairs $(y,\mathrm{z}) \in \left\{(m,\mathrm{ch}), (\kappa,\mathrm{f})\right\}$.

We then define $\mathcal{V}_{h,0}^{\mathrm{z}} = \left\{\phi_h \in \mathcal{V}_h^{\mathrm{z}} \big| \int_\Omega \phi_h \;dx = 0\right\}$ along with the parameter-dependent norm $\lVert \phi_h \rVert_{h,y} = \lVert y^{\frac{1}{2}} \nabla \phi_h \rVert_{\vecSpL}$. Moreover, let $\mathcal{V}_{h,y}^{\mathrm{z},'}$ denote the dual space of $\left(\mathcal{V}_{h,0}^{\mathrm{z}}, \lVert \cdot \rVert_{h,y}\right)$ with dual norm $\lVert \cdot \rVert_{\mathcal{V}_{h,y}^{\mathrm{z},'}}$.
For any $s_h \in \mathcal{V}_{h,y}^{\mathrm{z},'}$ we consider the unique solution $-\Delta_{h,y}^{-1} s_h := r_h \in \mathcal{V}_{h,0}^{\mathrm{z}}$ of the auxiliary problem
\begin{align*}
 \left(y \nabla r_h, \nabla l_h\right) = \left\langle s_h, l_h \right\rangle
\end{align*}
for all $l_h \in \mathcal{V}_{h,0}^{\mathrm{z}}$. Hence, we obtain
\begin{align*}
 \lVert s_h \rVert_{\mathcal{V}_{h,y}^{\mathrm{z},'}} := \sup_{\substack{l_h \in \mathcal{V}_{h,0}^{\mathrm{z}} \\ \lVert l_h \rVert_{h,y} \neq 0}} \frac{\left\langle s_h, l_h \right\rangle}{\lVert l_h \rVert_{h,y}} = \sup_{\substack{l_h \in \mathcal{V}_{h,0}^{\mathrm{z}} \\ \lVert l_h \rVert_{h,y} \neq 0}} \frac{\left(y \nabla r_h, \nabla l_h\right)}{\lVert y^{\frac{1}{2}} \nabla l_h \rVert_{\vecSpL}} = \lVert y^{\frac{1}{2}} \nabla r_h \rVert_{\vecSpL}.
\end{align*}

Furthermore, for $t_h, q_h \in \mathcal{V}_{h,0}^{\mathrm{z}}$ the inner-product on $\mathcal{V}_{h,y}^{\mathrm{z},'}$ is given by
\begin{align}
 \left(t_h, q_h\right)_{\mathcal{V}_{h,y}^{\mathrm{z},'}} = \left(v_h, q_h\right),
\label{eq:inner_dual_ch_f}
\end{align}
where $v_h \in \mathcal{V}_h^{\mathrm{z}}$ solves
\begin{align}
 \left(y \nabla v_h, \nabla l_h\right) = \left(t_h, l_h\right)
 \label{eq:aux_inner_dual_ch_f}
\end{align}
for all $l_h \in \mathcal{V}_h^{\mathrm{z}}$. Consequently, we have that
\begin{align*}
 \left(q_h, q_h\right)_{\mathcal{V}_{h,y}^{\mathrm{z},'}} = \left(v_h, q_h\right) = \left(y \nabla v_h, \nabla v_h\right) = \lVert y^{\frac{1}{2}} \nabla v_h \rVert_{\vecSpL}^2 = \lVert q_h \rVert_{\mathcal{V}_{h,y}^{\mathrm{z},'}}^2.
\end{align*}

\begin{remark} \label{rem:inner_dual_ch_f}
 Note that \eqref{eq:aux_inner_dual_ch_f} holds true for all $l_h \in \mathcal{V}_h^{\mathrm{z}}$ as $t_h \in \mathcal{V}_{h,0}^{\mathrm{z}}$. Moreover, uniqueness of $v_h$ solving \eqref{eq:aux_inner_dual_ch_f} can be imposed by prescribing its mean. Further note that the latter is still in consistency with \eqref{eq:inner_dual_ch_f} as the value of the inner-product $\left(v_h, q_h\right)$ does not change subject to altering the mean of $v_h$ as $q_h \in \mathcal{V}_{h,0}^{\mathrm{z}}$.
\end{remark}

\begin{lemma}
 For all $q_h \in \mathcal{V}_{h,0}^{\mathrm{z}}$ (with $\|q_h\|_{h,y} \neq 0$) there exists a positive constant $C_\mathrm{inv}$ such that
 \begin{align}
  \|q_h\|_{L^2(\Omega)} \leq \frac{C_\mathrm{inv}}{h} \lVert q_h \rVert_{\mathcal{V}_{h,y}^{\mathrm{z},'}}.
\label{eq:inv_eq_dual}
 \end{align}
\end{lemma}
\begin{proof}
The inverse inequality
\begin{align*}
 \|q_h\|_{H^1(\Omega)} \leq C h^{-1} \|q_h\|_{L^2(\Omega)}
\end{align*}
for some constant $C > 0$ (see e.g.~\cite{Ern2004}, Corollary 1.141; \cite{Brenner2008}, Theorem 4.5.11) implies
\begin{align}
 \|\nabla q_h\|_{\vecSpL} \leq C h^{-1} \|q_h\|_{L^2(\Omega)}.
\label{eq:inv_eq}
\end{align}
For $q_h \in \mathcal{V}_{h,0}^{\mathrm{z}}$ with $\|q_h\|_{h,y} \neq 0$ by the definition of the dual norm we have
\begin{align*}
 \lVert q_h \rVert_{\mathcal{V}_{h,y}^{\mathrm{z},'}} = \sup_{\substack{l_h \in \mathcal{V}_{h,0}^{\mathrm{z}} \\ \lVert l_h \rVert_{h,y} \neq 0}} \frac{\left\langle q_h, l_h \right\rangle}{\lVert l_h \rVert_{h,y}} \geq \frac{\left\langle q_h, q_h \right\rangle}{\lVert q_h \rVert_{h,y}} = \frac{\|q_h\|_{L^2(\Omega)}^2}{\lVert y^{\frac{1}{2}} \nabla q_h \rVert_{\vecSpL}},
\end{align*}
which is equivalent to
\begin{align*}
 y^{\frac{1}{2}} \|\nabla q_h\|_{\vecSpL} \lVert q_h \rVert_{\mathcal{V}_{h,y}^{\mathrm{z},'}} \geq \|q_h\|_{L^2(\Omega)}^2
\end{align*}
and by \eqref{eq:inv_eq} together with choosing $C_\mathrm{inv} = C y^{\frac{1}{2}}$ we obtain \eqref{eq:inv_eq_dual}.
\end{proof}

\subsection{State-independent material parameters} \label{sec:state-independent}

In this section, we limit the scope to the case where all of the material parameters are constant for a cleaner analysis of the problem. In particular, the fluid compressibility coefficient $M(\pf) = M$ and Biot-Willis coupling coefficient $\alpha(\pf) = \alpha$ are positive constants. We further consider the case of homogeneous elasticity, i.e., the elastic stiffness tensor $\mathbb{C}(\pf) = \mathbb{C}$ is constant, symmetric and positive definite. Hence, the variational derivatives (with respect to $\pf$) reduce to
\begin{align*}
 \delta_\pf\mathcal{E}_\mathrm{e}(\pf, \strain) = - \mathcal{T}'\!(\pf)\!:\!\mathbb{C}\big(\bm\varepsilon\!\left(\bm u\right)-\mathcal{T}(\pf)\big) 
 \quad \text{and} \quad
 \delta_\pf\mathcal{E}_{\mathrm{f}}(\pf, \strain, \theta) = 0.
\end{align*}

\subsubsection{Discretization}

Let $\tau$ be the uniform time-step size, and $n$ the time-step index.
The fully discrete problem corresponding to the variational formulation~\eqref{eq:model_var} then reads: \\
Given $\pf_h^{n-1} \in \mathcal{V}_h^{\mathrm{ch}}$ and $\theta_h^{n-1} \in \mathcal{V}_h^{\mathrm{f}}$, find $\pf_h^{n}, \mu_h^{n} \in \mathcal{V}_h^{\mathrm{ch}}$, $\bm u_h^{n} \in \mathcal{V}_h^{\bm u}$, and $\theta_h^{n}, p_h^n \in \mathcal{V}_h^{\mathrm{f}}$ such that
\begin{subequations}
\label{eq:model_disc}
\begin{align}
\left(\pf_h^{n} - \pf_h^{n-1},\eta_h^\pf\right) + \tau \left(m \nabla \mu_h^{n}, \nabla \eta_h^\pf\right) - \tau \left(R,\eta_h^\pf\right) &= 0, \label{eq:ch1split}\\
\left(\mu_h^{n},\eta_h^\mu\right) - \gamma\left(\ell\left(\nabla\pf_h^{n},\nabla\eta^\mu_h\right)+\frac{1}{\ell}\left(\Psi_c'\!\left(\pf_h^{n}\right)-\Psi_e'\!\left(\pf_h^{n-1}\right),\eta_h^\mu\right)\right)\nonumber\\ + \left(\mathcal{T}'\!(\pf_h^{n})\!:\!\mathbb{C}\big(\bm\varepsilon\!\left(\bm u_h^{n}\right)-\mathcal{T}(\pf_h^{n})\big),\eta_h^\mu\right) &= 0, \label{eq:ch2split}\\
\left(\mathbb{C}\!\left(\bm\varepsilon\!\left(\bm u_h^{n}\right)-\mathcal{T}\!\left(\pf_h^{n}\right)\right),\bm\varepsilon\!\left(\bm \eta_h^{\bm u}\right)\right) - \left(\alpha M \left(\theta_h^n - \alpha \Div \bm u_h^n\right), \Div \bm \eta_h^{\bm u}\right)-\left({\bm f},\bm \eta_h^{\bm u}\right) &= 0, \label{eq:elasticitysplit}\\
\left(\theta_h^{n} - \theta_h^{n-1}, \eta_h^\theta\right) + \tau \left(\kappa \nabla p_h^n, \nabla \eta_h^\theta\right) - \tau \left(S_\mathrm{f}, \eta_h^\theta\right) &= 0, \label{eq:flowsplit}\\
\left(p_h^n, \eta_h^p\right) - \left(M \left(\theta_h^n - \alpha \Div \bm u_h^n\right), \eta_h^p\right) &= 0 \label{eq:pressuresplit}
\end{align}
\end{subequations}
for all $\eta_h^\pf, \eta_h^\mu \in \mathcal{V}_h^\mathrm{ch},$ $\bm \eta_h^{\bm u} \in \mathcal{V}_h^{\bm u}$, and $\eta_h^\theta, \eta_h^{p} \in \mathcal{V}_h^\mathrm{f}$.

\begin{lemma}
The solution to the discrete problem~\eqref{eq:model_disc} is equivalent to the solution of the following minimization problem: \\
Given $\pf_h^{n-1} \in \mathcal{V}_h^\mathrm{ch}$, $\bm u_h^{n-1} \in \mathcal{V}_h^{\bm u}$ and $\theta_h^{n-1} \in \mathcal{V}_h^{\mathrm{f}}$, solve
\begin{align}
 \left(\pf_h^n, \bm u_h^n, \theta_h^n\right) = \argmin_{\pf_h \in \mathcal{\bar V}_h^{\mathrm{ch},n}, \bm u_h \in \mathcal{V}_h^{\bm u}, \theta_h \in \mathcal{\bar V}_h^{\mathrm{f},n}} \calH_\tau^n\!\left(\pf_h, \bm u_h, \theta_h\right)
\label{eq:minim}
\end{align}
with spaces
\begin{align*}
 \mathcal{\bar V}_h^{\mathrm{ch},n} := \left\{\pf_h \in \mathcal{V}_h^\mathrm{ch} \bigg| \int_\Omega \frac{\pf_h - \pf_h^{n-1}}{\tau} \;dx = \int_\Omega R \;dx\right\}
\quad \text{and} \quad 
 \mathcal{\bar V}_h^{\mathrm{f},n} := \left\{\theta_h \in \mathcal{V}_h^{\theta} \bigg| \int_\Omega \frac{\theta_h - \theta_h^{n-1}}{\tau} \;dx = \int_\Omega S_\mathrm{f} \;dx\right\}
\end{align*}
and the potential
\begin{align*}
 \calH_\tau^n\!\left(\pf_h, \bm u_h, \theta_h\right) := &\frac{\lVert \pf_h - \pf_h^{n-1} - \tau R \rVert_{\mathcal{V}_{h,m}^{\mathrm{ch},'}}^2}{2\tau} + \energy_\mathrm{c}(\pf_h, \strain[u_h], \theta_h) + \frac{\lVert \theta_h - \theta_h^{n-1} - \tau S_\mathrm{f} \rVert_{\mathcal{V}_{h,\kappa}^{\mathrm{f},'}}^2}{2\tau} \\
 &- \frac{\gamma}{\ell} \left(\Psi_e'\!\left(\pf_h^{n-1}\right), \pf_h\right) - \left({\bm f}, \bm u_h\right)
\end{align*}
with convex energy
\begin{align*}
 \energy_\mathrm{c}(\pf_h, \strain[u_h], \theta_h) := \energy_\mathrm{quad}(\pf_h, \strain[u_h], \theta_h) + \frac{\gamma}{\ell} \int_\Omega \Psi_c(\pf_h) \;dx,
\end{align*}
where
\begin{align*}
 \energy_\mathrm{quad}(\pf_h, \strain[u_h], \theta_h) := \int_\Omega \frac{\gamma \ell}{2} |\nabla \pf_h|^2 + \frac{1}{2} \big(\bm\varepsilon(\bm u_h) - \mathcal{T}(\pf_h)\big) \!:\!\mathbb{C} \big(\bm\varepsilon(\bm u_h) - \mathcal{T}(\pf_h)\big) + \frac{M}{2} (\theta_h - \alpha \Div \bm u_h)^2 \;dx.
\end{align*}
\label{lem:equivalence}
\end{lemma}
\begin{proof}
Firstly, we calculate the variational derivatives of the potential $\calH_\tau^n$ with respect to its arguments, yielding
\begin{align*}
 &\delta_\pf \calH_\tau^n\!\left(\pf_h^n, \bm u_h^n, \theta_h^n\right) = \frac{\pf_h^n - \pf_h^{n-1} - \tau R}{\tau} + \delta_\pf \energy_\mathrm{c}\!\left(\pf_h^n, \strain[u_h^n], \theta_h^n\right) - \frac{\gamma}{\ell} \Psi_e'\!\left(\pf_h^{n-1}\right), \\
 &\delta_{\bm u} \calH_\tau^n\!\left(\pf_h^n, \bm u_h^n, \theta_h^n\right) = \delta_{\strain} \energy_\mathrm{c}\!\left(\pf_h^n, \strain[u_h^n], \theta_h^n\right) - \bm f, \\
 &\delta_\theta \calH_\tau^n\!\left(\pf_h^n, \bm u_h^n, \theta_h^n\right) = \frac{\theta_h^n - \theta_h^{n-1} - \tau S_\mathrm{f}}{\tau} + \delta_\theta \energy_\mathrm{c}\!\left(\pf_h^n, \strain[u_h^n], \theta_h^n\right),
\end{align*}
where the respective variational derivatives of the convex energy $\energy_\mathrm{c}$ are given by
\begin{align*}
 \delta_\pf \energy_\mathrm{c}\!\left(\pf_h^n, \strain[u_h^n], \theta_h^n\right) &= \frac{\gamma}{\ell} \Psi_c'(\pf_h^n) - \gamma \ell \Delta \pf_h^n - \mathcal{T}'\!(\pf_h^n)\!:\!\mathbb{C}\big(\bm\varepsilon\!\left(\bm u_h^n\right)-\mathcal{T}(\pf_h^n)\big), \\
 \delta_{\strain} \energy_\mathrm{c}\!\left(\pf_h^n, \strain[u_h^n], \theta_h^n\right) &= \mathbb{C}\big(\bm\varepsilon\!\left(\bm u_h^n\right)-\mathcal{T}(\pf_h^n)\big) - \alpha M \left(\theta_h^n - \alpha \Div \bm u_h^n\right) \bm{I}, \\
 \delta_\theta \energy_\mathrm{c}\!\left(\pf_h^n, \strain[u_h^n], \theta_h^n\right) &= M \left(\theta_h^n - \alpha \Div \bm u_h^n\right).
\end{align*}
The optimality conditions to the minimization problem~\eqref{eq:minim} then read: \\
Find
$\pf_h^n \in \mathcal{\bar V}_h^{\mathrm{ch},n}$, $\bm u_h^n \in \mathcal{V}_h^{\bm u}$, $\theta_h^n \in \mathcal{\bar V}_h^{\mathrm{f},n}$
such that
\begin{align}
 &0 = \left\langle \delta_\pf \calH_\tau^n\!\left(\pf_h^n, \bm u_h^n, \theta_h^n\right), \pf_h\right\rangle = \left(\frac{\pf_h^n - \pf_h^{n-1}}{\tau} - R, \pf_h\right)_{\mathcal{V}_{h,m}^{\mathrm{ch},'}} + \left(\delta_\pf \energy_\mathrm{c}\!\left(\pf_h^n, \strain[u_h^n], \theta_h^n\right) - \frac{\gamma}{\ell} \Psi_e'\!\left(\pf_h^{n-1}\right), \pf_h\right), \label{eq:optim_ch}\\
 &0 = \left\langle \delta_{\bm u} \calH_\tau^n\!\left(\pf_h^n, \bm u_h^n, \theta_h^n\right), \bm u_h\right\rangle = \left(\delta_{\strain} \energy_\mathrm{c}\!\left(\pf_h^n, \strain[u_h^n], \theta_h^n\right), \bm\varepsilon\!\left(\bm u_h\right)\right) - \left(\bm f, \bm u_h\right), \label{eq:optim_el}\\
 &0 = \left\langle \delta_\theta \calH_\tau^n\!\left(\pf_h^n, \bm u_h^n, \theta_h^n\right), \theta_h\right\rangle = \left(\frac{\theta_h^n - \theta_h^{n-1}}{\tau} - S_\mathrm{f}, \theta_h\right)_{\mathcal{V}_{h,\kappa}^{\mathrm{f},'}} + \left(\delta_\theta \energy_\mathrm{c}\!\left(\pf_h^n, \strain[u_h^n], \theta_h^n\right), \theta_h\right) \label{eq:optim_f}
\end{align}
for all $\pf_h \in \mathcal{V}_{h,0}^{\mathrm{ch}}$, $\bm u_h \in \mathcal{V}_h^{\bm u}$, $\theta_h \in \mathcal{V}_{h,0}^{\mathrm{f}}$.

As $\pf_h^n \in \mathcal{\bar V}_h^{\mathrm{ch},n}$, we have $\frac{\pf_h^n - \pf_h^{n-1}}{\tau} - R \in \mathcal{V}_{h,0}^{\mathrm{ch}}$ and can thus employ the definition of $(\cdot,\cdot)_{\mathcal{V}_{h,m}^{\mathrm{ch},'}}$, cf.~\eqref{eq:inner_dual_ch_f}, yielding
\begin{align*}
 \left(\frac{\pf_h^n - \pf_h^{n-1}}{\tau} - R, \pf_h\right)_{\mathcal{V}_{h,m}^{\mathrm{ch},'}} = \left(- \mu_h^n, \pf_h\right),
\end{align*}
where $\mu_h^n \in \mathcal{V}_h^{\mathrm{ch}}$ solves
\begin{align}
 \left(m \nabla\left(- \mu_h^n\right), \nabla l_h\right) = \left(\frac{\pf_h^n - \pf_h^{n-1}}{\tau} - R, l_h\right)
 \label{eq:optim_ch1}
\end{align}
for all $l_h \in \mathcal{V}_h^{\mathrm{ch}}$ with mean fixed as
\begin{align}
 \frac{1}{|\Omega|} \int_\Omega \mu_h^n \;dx = \frac{1}{|\Omega|} \int_\Omega \delta_\pf \energy_\mathrm{c}\!\left(\pf_h^n, \strain[u_h^n], \theta_h^n\right) - \frac{\gamma}{\ell} \Psi_e'\!\left(\pf_h^{n-1}\right) \;dx
 \label{eq:optim_mean_mu}
\end{align}
in accordance with Remark~\ref{rem:inner_dual_ch_f}.
Hence, \eqref{eq:optim_ch} is equivalent to
\begin{align}
 0 = - \left(\mu_h^n, \pf_h\right) + \left(\delta_\pf \energy_\mathrm{c}\!\left(\pf_h^n, \strain[u_h^n], \theta_h^n\right) - \frac{\gamma}{\ell} \Psi_e'\!\left(\pf_h^{n-1}\right), \pf_h\right)
 \label{eq:optim_ch2}
\end{align}
for all $\pf_h \in \mathcal{V}_{h,0}^{\mathrm{ch}}$. Note that \eqref{eq:optim_ch2} is satisfied for all $\pf_h \in \mathcal{V}_h^{\mathrm{ch}}$ due to \eqref{eq:optim_mean_mu}. 

Similarly, as $\theta_h^n \in \mathcal{\bar V}_h^{\mathrm{f},n}$, we have $\frac{\theta_h^n - \theta_h^{n-1}}{\tau} - S_\mathrm{f} \in \mathcal{V}_{h,0}^{\mathrm{f}}$ and thus obtain
\begin{align*}
 \left(\frac{\theta_h^n - \theta_h^{n-1}}{\tau} - S_\mathrm{f}, \theta_h\right)_{\mathcal{V}_{h,\kappa}^{\mathrm{f},'}} = (- p_h^n, \theta_h)    
\end{align*}
from \eqref{eq:inner_dual_ch_f}, where $p_h^n \in \mathcal{V}_h^\mathrm{f}$ solves
\begin{align}
 \left(\kappa \nabla\left(- p_h^n\right), \nabla l_h\right) = \left(\frac{\theta_h^n - \theta_h^{n-1}}{\tau} - S_\mathrm{f}, l_h\right)
 \label{eq:optim_flow}
\end{align}
for all $l_h \in \mathcal{V}_h^{\mathrm{f}}$ with mean fixed as
\begin{align}
 \frac{1}{|\Omega|} \int_\Omega p_h^n \;dx = \frac{1}{|\Omega|} \int_\Omega \delta_\theta \energy_\mathrm{c}\!\left(\pf_h^n, \strain[u_h^n], \theta_h^n\right) \;dx
 \label{eq:optim_mean_p}
\end{align}
in accordance with Remark~\ref{rem:inner_dual_ch_f}.
Hence, \eqref{eq:optim_f} is equivalent to
\begin{align}
 0 = - \left(p_h^n, \theta_h\right) + \left(\delta_\theta \energy_\mathrm{c}\!\left(\pf_h^n, \strain[u_h^n], \theta_h^n\right), \theta_h\right)
 \label{eq:optim_pressure}
\end{align}
for all $\theta_h \in \mathcal{V}_{h,0}^{\mathrm{f}}$. Note that \eqref{eq:optim_pressure} is satisfied for all $\theta_h \in \mathcal{V}_h^{\mathrm{f}}$ due to \eqref{eq:optim_mean_p}.

Finally, we can infer that the solutions to \eqref{eq:optim_ch1}, \eqref{eq:optim_ch2}, \eqref{eq:optim_el}, \eqref{eq:optim_flow} and \eqref{eq:optim_pressure} are equivalent to the solutions to the discrete problem \eqref{eq:ch1split}--\eqref{eq:pressuresplit}.
\end{proof}

\begin{remark}
 The spaces $\mathcal{\bar V}_h^{\mathrm{ch},n}$ and $\mathcal{\bar V}_h^{\mathrm{f},n}$ are affine. Let $\mathrm{z} \in \{\mathrm{ch}, \mathrm{f}\}$. Then, for arbitrary $\vartheta_h^1, \vartheta_h^2 \in \mathcal{\bar V}_h^{\mathrm{z},n}$ the difference satisfies $\vartheta_h^1 - \vartheta_h^2 \in \mathcal{V}_{h,0}^{\mathrm{z}}$.
\end{remark}

\begin{lemma}
Under the absence of external contributions, i.e., $R = 0$, $\bm f = 0$ and $S_\mathrm{f} = 0$, the discretization scheme \eqref{eq:model_disc} satisfies the energy dissipation property
\begin{align*}
 \energy\!\left(\pf_h^n, \strain[u_h^n], \theta_h^n\right) \leq \energy\!\left(\pf_h^{n-1}, \strain[u_h^{n-1}], \theta_h^{n-1}\right)
\end{align*}
with free energy
\begin{align*}
 \energy(\pf, \strain, \theta) = \int_\Omega \gamma \left(\frac{1}{\ell} \Psi(\pf) + \frac{\ell}{2} |\nabla \pf|^2\right) + \frac{1}{2} \big(\bm\varepsilon(\bm u) - \mathcal{T}(\pf)\big) \!:\!\mathbb{C} \big(\bm\varepsilon(\bm u) - \mathcal{T}(\pf)\big) + \frac{M}{2} (\theta - \alpha \Div \bm u)^2 \;dx.
\end{align*}
\label{lem:disc_ener_diss}
\end{lemma}
\begin{proof}
As $R = 0$ and $S_\mathrm{f} = 0$, we obtain that $\pf_h^{n-1} \in \mathcal{\bar V}_h^{\mathrm{ch},n}$ and $\theta_h^{n-1} \in \mathcal{\bar V}_h^{\mathrm{f},n}$ respectively. Hence, by the definition of the minimization problem \eqref{eq:minim} we have
\begin{align*}
 \calH_\tau^n\!\left(\pf_h^n, \bm u_h^n, \theta_h^n\right) \leq \calH_\tau^n\!\left(\pf_h^{n-1}, \bm u_h^{n-1}, \theta_h^{n-1}\right),
\end{align*}
which is equivalent to
\begin{align*}
 &\frac{\lVert \pf_h^n - \pf_h^{n-1}\rVert_{\mathcal{V}_{h,m}^{\mathrm{ch},'}}^2}{2\tau} + \energy_\mathrm{c}(\pf_h^n, \strain[u_h^n], \theta_h^n) + \frac{\lVert \theta_h^n - \theta_h^{n-1}\rVert_{\mathcal{V}_{h,\kappa}^{\mathrm{f},'}}^2}{2\tau} - \frac{\gamma}{\ell} \left(\Psi_e'\!\left(\pf_h^{n-1}\right), \pf_h^n\right) \\
 &\qquad \leq \energy_\mathrm{c}(\pf_h^{n-1}, \strain[u_h^{n-1}], \theta_h^{n-1}) - \frac{\gamma}{\ell} \left(\Psi_e'\!\left(\pf_h^{n-1}\right), \pf_h^{n-1}\right).
\end{align*}
By the convexity of $\Psi_e$ we further have
\begin{align*}
 \Psi_e'\!\left(\pf_h^{n-1}\right) \left(\pf_h^n - \pf_h^{n-1}\right) \leq \Psi_e\!\left(\pf_h^n\right) - \Psi_e\!\left(\pf_h^{n-1}\right)
\end{align*}
and since the free energy can be decomposed as
\begin{align*}
 \energy(\pf, \strain, \theta) = \energy_\mathrm{c}(\pf, \strain, \theta) - \frac{\gamma}{\ell} \int_\Omega \Psi_e(\pf) \;dx
\end{align*}
we obtain
\begin{align*}
 \frac{\lVert \pf_h^n - \pf_h^{n-1}\rVert_{\mathcal{V}_{h,m}^{\mathrm{ch},'}}^2}{2\tau} + \energy(\pf_h^n, \strain[u_h^n], \theta_h^n) + \frac{\lVert \theta_h^n - \theta_h^{n-1}\rVert_{\mathcal{V}_{h,\kappa}^{\mathrm{f},'}}^2}{2\tau} \leq \energy(\pf_h^{n-1}, \strain[u_h^{n-1}], \theta_h^{n-1})
\end{align*}
and thus
\begin{align*}
 \energy\!\left(\pf_h^n, \strain[u_h^n], \theta_h^n\right) \leq \energy\!\left(\pf_h^{n-1}, \strain[u_h^{n-1}], \theta_h^{n-1}\right)
\end{align*}
for all $\tau$ and $n$. Due to the equivalence of the minimization problem \eqref{eq:minim} and the discretization scheme \eqref{eq:model_disc}, cf. Lemma~\ref{lem:equivalence}, we conclude the proof.
\end{proof}

\subsubsection{Alternating minimization} \label{sec:alter_min}

The iterations of the alternating minimization scheme are initialized by using the solutions at the previous time step
\begin{align*}
 \pf_h^{n,0} = \pf_h^{n-1}, \quad \bm u_h^{n,0} = \bm u_h^{n-1}, \quad \theta_h^{n,0} = \theta_h^{n-1},
\end{align*}
and subsequently the potential $\calH_\tau^n$ is minimized
\begin{align*}
 \pf_h^{n,i} &= \argmin_{\pf_h \in \mathcal{\bar V}_h^{\mathrm{ch},n}} \calH_\tau^n\!\left(\pf_h, \bm u_h^{n,i-1}, \theta_h^{n,i-1}\right) \\
 \left(\bm u_h^{n,i}, \theta_h^{n,i}\right) &= \argmin_{\bm u_h \in \mathcal{V}_h^{\bm u}, \theta_h \in \mathcal{\bar V}_h^{\mathrm{f},n}} \calH_\tau^n\!\left(\pf_h^{n,i}, \bm u_h, \theta_h\right)
\end{align*}
with iteration index $i$. We can then formulate the corresponding iterative scheme in iteration $i$ based on the variational system of equations as: \\
Given $\pf_h^{n-1} \in \mathcal{V}_h^{\mathrm{ch}}$, $\bm u_h^{n,i-1} \in \mathcal{V}_h^{\bm u}$, and $\theta_h^{n-1} \in \mathcal{V}_h^{\mathrm{f}}$, find $\pf_h^{n,i}, \mu_h^{n,i} \in \mathcal{V}_h^{\mathrm{ch}}$, $\bm u_h^{n,i} \in \mathcal{V}_h^{\bm u}$, and $\theta_h^{n,i}, p_h^{n,i} \in \mathcal{V}_h^{\mathrm{f}}$ such that
\begin{subequations}
\label{eq:model_disc_alter}
\begin{align}
\left(\pf_h^{n,i} - \pf_h^{n-1},\eta_h^\pf\right) + \left( m \nabla \mu_h^{n,i}, \nabla \eta_h^\pf\right)- \tau \left(R,\eta_h^\pf\right) &= 0, \label{eq:ch1split_alter}\\
\left(\mu_h^{n,i},\eta_h^\mu\right) - \gamma\left(\ell\left(\nabla\pf_h^{n,i},\nabla\eta^\mu_h\right)+\frac{1}{\ell}\left(\Psi_c'\!\left(\pf_h^{n,i}\right)-\Psi_e'\!\left(\pf_h^{n-1}\right),\eta_h^\mu\right)\right)\nonumber\\ + \left(\mathcal{T}'\!(\pf_h^{n,i})\!:\!\mathbb{C}\!\left(\bm\varepsilon\!\left(\bm u_h^{n,i-1}\right)-\mathcal{T}(\pf_h^{n,i})\right),\eta_h^\mu\right) &= 0, \label{eq:ch2split_alter}\\
\left(\mathbb{C}\!\left(\bm\varepsilon\!\left(\bm u_h^{n,i}\right)-\mathcal{T}\!\left(\pf_h^{n,i}\right)\right),\bm\varepsilon\!\left(\bm \eta_h^{\bm u}\right)\right) - \left(\alpha M \left(\theta_h^{n,i} - \alpha \Div \bm u_h^{n,i}\right), \Div \bm \eta_h^{\bm u}\right)-\left({\bm f},\bm \eta_h^{\bm u}\right) &= 0, \label{eq:elasticitysplit_alter}\\
\left(\theta_h^{n,i} - \theta_h^{n-1}, \eta_h^\theta\right) + \tau \left(\kappa \nabla p_h^{n,i}, \nabla \eta_h^\theta\right) - \tau \left(S_\mathrm{f}, \eta_h^\theta\right) &= 0, \label{eq:flowsplit_alter}\\
\left(p_h^{n,i}, \eta_h^p\right) - \left(M \left(\theta_h^{n,i} - \alpha \Div \bm u_h^{n,i}\right), \eta_h^p\right) &= 0 \label{eq:pressuresplit_alter}
\end{align}
\end{subequations}
for all $\eta_h^\pf, \eta_h^\mu \in \mathcal{V}_h^\mathrm{ch},$ $\bm \eta_h^{\bm u} \in \mathcal{V}_h^{\bm u}$, and $\eta_h^\theta, \eta_h^p \in \mathcal{V}_h^\mathrm{f}$. Note that we can use the spaces $\mathcal{V}_h^\mathrm{ch}$ and $\mathcal{V}_h^\mathrm{f}$ instead of the restricted spaces $\mathcal{\bar V}_h^{\mathrm{ch},n}$ and $\mathcal{\bar V}_h^{\mathrm{f},n}$ here following from the same canonical extensions as in the proof of Lemma~\ref{lem:equivalence}.

The proof of convergence for the alternating minimization scheme \eqref{eq:model_disc_alter} is based on the abstract theory presented in \cite{Both2022}. Lemma~\ref{lem:aux_conv_result} in the appendix recapitulates the appropriate result in terms of the present article.

\begin{theorem}
 The alternating minimization scheme~\eqref{eq:model_disc_alter} converges linearly in the sense that
\begin{align}
 \calH_\tau^n\!\left(\pf_h^{n,i},\bm u_h^{n,i},\theta_h^{n,i}\right) - \calH_\tau^n\!\left(\pf_h^{n},\bm u_h^{n},\theta_h^{n}\right) \leq \left(1 - \frac{\beta_\mathrm{ch}}{L_\mathrm{ch}}\right) \left(1 - \beta_\mathrm{b}\right) \Big( \calH_\tau^n\!\left(\pf_h^{n,i-1},\bm u_h^{n,i-1},\theta_h^{n,i-1}\right) - \calH_\tau^n\!\left(\pf_h^{n},\bm u_h^{n},\theta_h^{n}\right) \Big),
 \label{eq:conv_result}
\end{align}
where $\beta_\mathrm{ch} = \beta_\mathrm{b} = 1 - \left(\frac{h^2}{\tau C_\mathrm{inv}^2 C_\mathbb{C} C_\mathcal{T}^2} + \frac{\gamma \ell}{C_\Omega^2 C_\mathbb{C} C_\mathcal{T}^2} + 1\right)^{-1}$, and $L_\mathrm{ch} = 1 + \frac{\gamma}{\ell} L_{\Psi_\mathrm{c}'} \left(\frac{1}{\tau} \frac{h^2}{C_\mathrm{inv}^2} + \gamma \ell \frac{1}{C_\Omega^2} + c_\mathbb{C} c_\mathcal{T}^2\right)^{-1}$.
\label{thm:conv_result}
\end{theorem}
\begin{proof}
We define the (semi-)norms
\begin{align*}
 \lVert (\pf_h, \bm u_h, \theta_h) \rVert^2 &:= \frac{\lVert \pf_h \rVert_{\mathcal{V}_{h,m}^{\mathrm{ch},'}}^2}{\tau} + \gamma \ell \lVert \nabla\pf_h \rVert_{\vecSpL}^2 + \left(\mathbb{C}\!\left(\bm\varepsilon\!\left(\bm u_h\right)-\mathcal{T}\!\left(\pf_h\right)\right),\bm\varepsilon\!\left(\bm u_h\right)-\mathcal{T}\!\left(\pf_h\right)\right) \\ 
 &\qquad+ \frac{\lVert \theta_h \rVert_{\mathcal{V}_{h,\kappa}^{\mathrm{f},'}}^2}{\tau} + \left(M \left(\theta_h - \alpha \Div \bm u_h\right), \theta_h - \alpha \Div \bm u_h\right), \\
 \lVert \pf_h \rVert_\mathrm{ch}^2 &:= \frac{\lVert \pf_h \rVert_{\mathcal{V}_{h,m}^{\mathrm{ch},'}}^2}{\tau} + \gamma \ell \lVert \nabla\pf_h \rVert_{\vecSpL}^2 + \left(\mathbb{C}\mathcal{T}\!\left(\pf_h\right), \mathcal{T}\!\left(\pf_h\right)\right), \\
 \lVert (\bm u_h, \theta_h) \rVert_\mathrm{b}^2 &:= \left(\mathbb{C}\bm\varepsilon\!\left(\bm u_h\right), \bm\varepsilon\!\left(\bm u_h\right)\right) + \frac{\lVert \theta_h \rVert_{\mathcal{V}_{h,\kappa}^{\mathrm{f},'}}^2}{\tau} + \left(M \left(\theta_h - \alpha \Div \bm u_h\right), \theta_h - \alpha \Div \bm u_h\right)
\end{align*}
for $\pf_h \in \mathcal{V}_{h,0}^\mathrm{ch}$, $\bm u_h \in \mathcal{V}_h^{\bm u}$ and $\theta_h \in \mathcal{V}_{h,0}^{\mathrm{f}}$.
Firstly, we show the relations \eqref{eq:rel_norm_ch} and \eqref{eq:rel_norm_b} between the norms. We have
\begin{align*}
 \lVert (\pf_h, \bm u_h, \theta_h) \rVert^2 
 &= \frac{\lVert \pf_h \rVert_{\mathcal{V}_{h,m}^{\mathrm{ch},'}}^2}{\tau} + \gamma \ell \lVert \nabla\pf_h \rVert_{\vecSpL}^2 + \left(\mathbb{C}\bm\varepsilon\!\left(\bm u_h\right), \bm\varepsilon\!\left(\bm u_h\right)\right) - 2 \left(\mathbb{C}\bm\varepsilon\!\left(\bm u_h\right), \mathcal{T}\!\left(\pf_h\right)\right) + \left(\mathbb{C}\mathcal{T}\!\left(\pf_h\right), \mathcal{T}\!\left(\pf_h\right)\right) \\
 &\qquad+ \frac{\lVert \theta_h \rVert_{\mathcal{V}_{h,\kappa}^{\mathrm{f},'}}^2}{\tau} + \left(M \left(\theta_h - \alpha \Div \bm u_h\right), \theta_h - \alpha \Div \bm u_h\right)
\end{align*}
and by the Cauchy-Schwarz-type inequality, Young's inequality and the boundedness of $\mathcal{T}(\cdot)$ we get
\begin{align*}
 2 \left(\mathbb{C}\bm\varepsilon\!\left(\bm u_h\right), \mathcal{T}\!\left(\pf_h\right)\right) &\leq \lambda \left(\mathbb{C}\bm\varepsilon\!\left(\bm u_h\right), \bm\varepsilon\!\left(\bm u_h\right)\right) + \frac{1}{\lambda} \left(c_1 + c_2 + c_3\right) \left(\mathbb{C}\mathcal{T}\!\left(\pf_h\right), \mathcal{T}\!\left(\pf_h\right)\right) \\
 &\leq \lambda \left(\mathbb{C}\bm\varepsilon\!\left(\bm u_h\right), \bm\varepsilon\!\left(\bm u_h\right)\right) + \frac{1}{\lambda} (c_1 + c_2) C_\mathbb{C} C_\mathcal{T}^2 \lVert \pf_h \rVert_{L^2(\Omega)}^2 + \frac{1}{\lambda} c_3 \left(\mathbb{C}\mathcal{T}\!\left(\pf_h\right), \mathcal{T}\!\left(\pf_h\right)\right)
\end{align*}
with coefficients $0 \leq c_k \leq 1$ for $k \in \{1,2,3\}$, $c_1 + c_2 + c_3 = 1$, and $\lambda > 0$ to be chosen later.
We employ \eqref{eq:inv_eq_dual} and the Poincar\'e-Wirtinger inequality with Poincar\'e constant $C_\Omega$ to obtain
\begin{align*}
 2 \left(\mathbb{C}\bm\varepsilon\!\left(\bm u_h\right), \mathcal{T}\!\left(\pf_h\right)\right) &\leq \lambda \left(\mathbb{C}\bm\varepsilon\!\left(\bm u_h\right), \bm\varepsilon\!\left(\bm u_h\right)\right) + \frac{1}{\lambda} c_1 \frac{C_\mathrm{inv}^2}{h^2} C_\mathbb{C} C_\mathcal{T}^2 \lVert \pf_h \rVert_{\mathcal{V}_{h,m}^{\mathrm{ch},'}}^2 \\
 &\qquad + \frac{1}{\lambda} c_2 C_\Omega^2 C_\mathbb{C} C_\mathcal{T}^2 \lVert \nabla \pf_h \rVert_{\vecSpL} + \frac{1}{\lambda} c_3 \left(\mathbb{C}\mathcal{T}\!\left(\pf_h\right), \mathcal{T}\!\left(\pf_h\right)\right)
\end{align*}
and thus
\begin{align}
 \lVert (\pf_h, \bm u_h, \theta_h) \rVert^2 &\geq
 \left(\frac{1}{\tau} - \frac{1}{\lambda} c_1 \frac{C_\mathrm{inv}^2}{h^2} C_\mathbb{C} C_\mathcal{T}^2\right) \lVert \pf_h \rVert_{\mathcal{V}_{h,m}^{\mathrm{ch},'}}^2 + \left(\gamma \ell - \frac{1}{\lambda} c_2 C_\Omega^2 C_\mathbb{C} C_\mathcal{T}^2\right) \lVert \nabla \pf_h \rVert_{\vecSpL} \nonumber\\
 &\quad + \left(1 - \frac{c_3}{\lambda}\right) \left(\mathbb{C}\mathcal{T}\!\left(\pf_h\right), \mathcal{T}\!\left(\pf_h\right)\right) + \left(1 - \lambda\right) \left(\mathbb{C}\bm\varepsilon\!\left(\bm u_h\right), \bm\varepsilon\!\left(\bm u_h\right)\right) \nonumber\\
 &\qquad + \frac{\lVert \theta_h \rVert_{\mathcal{V}_{h,\kappa}^{\mathrm{f},'}}^2}{\tau} + \left(M \left(\theta_h - \alpha \Div \bm u_h\right), \theta_h - \alpha \Div \bm u_h\right).
 \label{eq:bound_full_norm}
\end{align}
In order to obtain the bound \eqref{eq:rel_norm_ch} we drop the last two quadratic contributions in \eqref{eq:bound_full_norm} yielding
\begin{align*}
 \lVert (\pf_h, \bm u_h, \theta_h) \rVert^2 &\geq
 \left(\frac{1}{\tau} - \frac{1}{\lambda} c_1 \frac{C_\mathrm{inv}^2}{h^2} C_\mathbb{C} C_\mathcal{T}^2\right) \lVert \pf_h \rVert_{\mathcal{V}_{h,m}^{\mathrm{ch},'}}^2 + \left(\gamma \ell - \frac{1}{\lambda} c_2 C_\Omega^2 C_\mathbb{C} C_\mathcal{T}^2\right) \lVert \nabla \pf_h \rVert_{\vecSpL} \\
 &\quad + \left(1 - \frac{c_3}{\lambda}\right) \left(\mathbb{C}\mathcal{T}\!\left(\pf_h\right), \mathcal{T}\!\left(\pf_h\right)\right) + \left(1 - \lambda\right) \left(\mathbb{C}\bm\varepsilon\!\left(\bm u_h\right), \bm\varepsilon\!\left(\bm u_h\right)\right)
\end{align*}
and subsequently choose
\begin{align*}
 \lambda = 1, \quad c_1 = (1 - \beta_\mathrm{ch}) \frac{h^2}{\tau C_\mathrm{inv}^2 C_\mathbb{C} C_\mathcal{T}^2}, \quad c_2 = (1 - \beta_\mathrm{ch}) \frac{\gamma \ell}{C_\Omega^2 C_\mathbb{C} C_\mathcal{T}^2}, \quad c_3 = 1 - \beta_\mathrm{ch}
\end{align*}
along with
\begin{align*}
 \beta_\mathrm{ch} = 1 - \left(\frac{h^2}{\tau C_\mathrm{inv}^2 C_\mathbb{C} C_\mathcal{T}^2} + \frac{\gamma \ell}{C_\Omega^2 C_\mathbb{C} C_\mathcal{T}^2} + 1\right)^{-1}.
\end{align*}
Similarly, in \eqref{eq:bound_full_norm} we choose
\begin{align*}
 c_1 = \lambda \frac{h^2}{\tau C_\mathrm{inv}^2 C_\mathbb{C} C_\mathcal{T}^2}, \quad c_2 = \lambda \frac{\gamma \ell}{C_\Omega^2 C_\mathbb{C} C_\mathcal{T}^2}, \quad c_3 = \lambda, \quad \lambda = \left(\frac{h^2}{\tau C_\mathrm{inv}^2 C_\mathbb{C} C_\mathcal{T}^2} + \frac{\gamma \ell}{C_\Omega^2 C_\mathbb{C} C_\mathcal{T}^2} + 1\right)^{-1} 
\end{align*}
to obtain
\begin{align*}
 \lVert (\pf_h, \bm u_h, \theta_h) \rVert^2 &\geq
 \left(1 - \lambda\right) \left(\mathbb{C}\bm\varepsilon\!\left(\bm u_h\right), \bm\varepsilon\!\left(\bm u_h\right)\right) \\
 &\qquad + \frac{\lVert \theta_h \rVert_{\mathcal{V}_{h,\kappa}^{\mathrm{f},'}}^2}{\tau} + \left(M \left(\theta_h - \alpha \Div \bm u_h\right), \theta_h - \alpha \Div \bm u_h\right) \\
 &\geq (1 - \lambda) \lVert (\bm u_h, \theta_h) \rVert_\mathrm{b}^2
\end{align*}
and by choosing $\beta_\mathrm{b} = 1 - \lambda$ we get the desired bound \eqref{eq:rel_norm_b}.

Secondly, we prove the potential related bounds \eqref{eq:H_conv}, \eqref{eq:Lip_cont_ch} and \eqref{eq:Lip_cont_b}. We have
\begin{align*}
 &\Big(\delta_{(\pf,\bm u,\theta)} \calH_\tau^n\!\left(\pf_h^1, \bm u_h^1, \theta_h^1\right) - \delta_{(\pf,\bm u,\theta)} \calH_\tau^n\!\left(\pf_h^2, \bm u_h^2, \theta_h^2\right), \left(\pf_h^1 - \pf_h^2, \bm u_h^1 - \bm u_h^2, \theta_h^1 - \theta_h^2\right)\Big) \\
 &= \left(\frac{\pf_h^1 - \pf_h^2}{\tau}, \pf_h^1 - \pf_h^2\right)_{\mathcal{V}_{h,m}^{\mathrm{ch},'}} \\
 &\quad+ \Big(\delta_{(\pf,\strain,\theta)} \energy_\mathrm{quad}\!\left(\pf_h^1, \strain[u_h^1], \theta_h^1\right) - \delta_{(\pf,\strain,\theta)} \energy_\mathrm{quad}\!\left(\pf_h^2, \strain[u_h^2], \theta_h^2\right), \left(\pf_h^1 - \pf_h^2, \strain[u_h^1 - \bm u_h^2], \theta_h^1 - \theta_h^2\right)\Big) \\
 &\qquad+ \frac{\gamma}{\ell} \int_\Omega \left(\Psi_c'(\pf_h^1) - \Psi_c'(\pf_h^2)\right) \left(\pf_h^1 - \pf_h^2\right) dx + \left(\frac{\theta_h^1 - \theta_h^2}{\tau}, \theta_h^1 - \theta_h^2\right)_{\mathcal{V}_{h,\kappa}^{\mathrm{f},'}} \\
 &\geq \frac{\lVert \pf_h^1 - \pf_h^2 \rVert_{\mathcal{V}_{h,m}^{\mathrm{ch},'}}^2}{\tau} + 2 \energy_\mathrm{quad}\!\left(\pf_h^1 - \pf_h^2, \strain[u_h^1 - \bm u_h^2], \theta_h^1 - \theta_h^2\right) + \frac{\lVert \theta_h^1 - \theta_h^2 \rVert_{\mathcal{V}_{h,\kappa}^{\mathrm{f},'}}^2}{\tau} \\
 &= \lVert (\pf_h^1 - \pf_h^2, \bm u_h^1 - \bm u_h^2, \theta_h^1 - \theta_h^2) \rVert^2
\end{align*}
for all $\pf_h^1, \pf_h^2 \in \mathcal{\bar V}_h^{\mathrm{ch},n}$, $\bm u_h^1, \bm u_h^2 \in \mathcal{V}_h^{\bm u}$, and $\theta_h^1, \theta_h^2 \in \mathcal{\bar V}_h^{\mathrm{f},n}$, where we employ the convexity of $\Psi_c(\cdot)$ and the definition of $\energy_\mathrm{quad}$, yielding the desired strong convexity of $\calH_\tau^n$ \eqref{eq:H_conv} with constant $\sigma = 1$.

Moreover, we have
\begin{align*}
 &\Big(\delta_\pf \calH_\tau^n\!\left(\pf_h^1, \bm u_h, \theta_h\right) - \delta_\pf \calH_\tau^n\!\left(\pf_h^2, \bm u_h, \theta_h\right), \pf_h^1 - \pf_h^2\Big) \\
 &= \left(\frac{\pf_h^1 - \pf_h^2}{\tau}, \pf_h^1 - \pf_h^2\right)_{\mathcal{V}_{h,m}^{\mathrm{ch},'}} + \Big(\delta_\pf \energy_\mathrm{quad}\!\left(\pf_h^1, \bm u_h, \theta_h\right) - \delta_\pf \energy_\mathrm{quad}\!\left(\pf_h^2, \bm u_h, \theta_h\right), \pf_h^1 - \pf_h^2\Big) \\
 &\qquad+ \frac{\gamma}{\ell} \int_\Omega \left(\Psi_c'(\pf_h^1) - \Psi_c'(\pf_h^2)\right) \left(\pf_h^1 - \pf_h^2\right) dx \\
 &\leq \frac{\lVert \pf_h^1 - \pf_h^2 \rVert_{\mathcal{V}_{h,m}^{\mathrm{ch},'}}^2}{\tau} + 2 \energy_\mathrm{quad}\!\left(\pf_h^1 - \pf_h^2, \bm 0, 0\right) + \frac{\gamma}{\ell} L_{\Psi_\mathrm{c}'} \lVert \pf_h^1 - \pf_h^2 \rVert_{L^2(\Omega)}^2 \\
 &= \lVert \pf_h^1 - \pf_h^2 \rVert_\mathrm{ch}^2 + \frac{\gamma}{\ell} L_{\Psi_\mathrm{c}'} \lVert \pf_h^1 - \pf_h^2 \rVert_{L^2(\Omega)}^2
\end{align*}
for all $\pf_h^1, \pf_h^2 \in \mathcal{\bar V}_h^{\mathrm{ch},n}$, $\bm u_h \in \mathcal{V}_h^{\bm u}$, and $\theta_h \in \mathcal{\bar V}_h^{\mathrm{f},n}$, using the Lipschitz continuity of $\Psi_\mathrm{c}'(\cdot)$ and the definition of $\energy_\mathrm{quad}$. By \eqref{eq:inv_eq_dual}, the Poincar\'e-Wirtinger inequality and the boundedness of $\mathcal{T}(\cdot)$ we get
\begin{align*}
 \lVert \pf_h^1 -\pf_h^2 \rVert_\mathrm{ch}^2 &= \frac{\lVert \pf_h^1 - \pf_h^2 \rVert_{\mathcal{V}_{h,m}^{\mathrm{ch},'}}^2}{\tau} + \gamma \ell \lVert \nabla \pf_h^1 - \nabla \pf_h^2 \rVert_{\vecSpL}^2 + \left(\mathbb{C}\mathcal{T}\!\left(\pf_h^1 - \pf_h^2\right), \mathcal{T}\!\left(\pf_h^1 - \pf_h^2\right)\right) \\
 &\geq \left(\frac{1}{\tau} \frac{h^2}{C_\mathrm{inv}^2} + \gamma \ell \frac{1}{C_\Omega^2} + c_\mathbb{C} c_\mathcal{T}^2\right) \lVert \pf_h^1 - \pf_h^2 \rVert_{L^2(\Omega)}^2
\end{align*}
and hence
\begin{align*}
 \Big(\delta_\pf \calH_\tau^n\!\left(\pf_h^1, \bm u_h, \theta_h\right) - \delta_\pf \calH_\tau^n\!\left(\pf_h^2, \bm u_h, \theta_h\right), \pf_h^1 - \pf_h^2\Big) \leq L_\mathrm{ch} \lVert \pf_h^1 -\pf_h^2 \rVert_\mathrm{ch}^2
\end{align*}
with constant $L_\mathrm{ch} = 1 + \frac{\gamma}{\ell} L_{\Psi_\mathrm{c}'} \left(\frac{1}{\tau} \frac{h^2}{C_\mathrm{inv}^2} + \gamma \ell \frac{1}{C_\Omega^2} + c_\mathbb{C} c_\mathcal{T}^2\right)^{-1}$, i.e., \eqref{eq:Lip_cont_ch}. \\
Similarly, we have
\begin{align*}
 &\Big(\delta_{(\bm u,\theta)}\calH_\tau^n\!\left(\pf_h, \bm u_h^1, \theta_h^1\right) - \delta_{(\bm u,\theta)}\calH_\tau^n\!\left(\pf_h, \bm u_h^2, \theta_h^2\right), \left(\bm u_h^1 - \bm u_h^2, \theta_h^1 - \theta_h^2\right)\Big) \\
 &= \Big(\delta_{(\strain,\theta)} \energy_\mathrm{quad}\!\left(\pf_h, \strain[u_h^1], \theta_h^1\right) - \delta_{(\strain,\theta)} \energy_\mathrm{quad}\!\left(\pf_h, \strain[u_h^2], \theta_h^2\right), \left(\strain[u_h^1 - \bm u_h^2], \theta_h^1 - \theta_h^2\right)\Big) + \left(\frac{\theta_h^1 - \theta_h^2}{\tau}, \theta_h^1 - \theta_h^2\right)_{\mathcal{V}_{h,\kappa}^{\mathrm{f},'}} \\
 &= 2 \energy_\mathrm{quad}\!\left(0, \strain[u_h^1 - \bm u_h^2], \theta_h^1 - \theta_h^2\right) + \frac{\lVert \theta_h^1 - \theta_h^2 \rVert_{\mathcal{V}_{h,\kappa}^{\mathrm{f},'}}^2}{\tau} \\
 &= \lVert (\bm u_h^1 - \bm u_h^2, \theta_h^1 - \theta_h^2) \rVert_\mathrm{b}^2
\end{align*}
for all $\pf_h \in \mathcal{\bar V}_h^{\mathrm{ch},n}$, $\bm u_h^1, \bm u_h^2 \in \mathcal{V}_h^{\bm u}$, and $\theta_h^1, \theta_h^2 \in \mathcal{\bar V}_h^{\mathrm{f},n}$, where we use that $\mathcal{T}(0) = \bm 0$, yielding \eqref{eq:Lip_cont_b} with $L_\mathrm{b} = 1$.

Finally, employing Lemma~\ref{lem:aux_conv_result} we obtain the desired convergence result \eqref{eq:conv_result}.
\end{proof}

\subsection{State-dependent material parameters} \label{sec:state-dependent}
For state-dependent material parameters $M(\pf)$, $\alpha(\pf)$ and $\mathbb{C}(\pf)$ we introduce a semi-implicit discretization of \eqref{eq:model_var} that corresponds to a convex minimization problem, allowing one to similarly prove convergence of the corresponding alternating minimization scheme.

\subsubsection{Discretization}

The fully discrete problem corresponding to \eqref{eq:model_var} reads: \\
Given $\pf_h^{n-1} \in \mathcal{V}_h^{\mathrm{ch}}$ and $\theta_h^{n-1} \in \mathcal{V}_h^{\mathrm{f}}$, find $\pf_h^{n}, \mu_h^{n} \in \mathcal{V}_h^{\mathrm{ch}}$, $\bm u_h^{n} \in \mathcal{V}_h^{\bm u}$, and $\theta_h^{n}, p_h^n \in \mathcal{V}_h^{\mathrm{f}}$ such that
\begin{subequations}
\label{eq:model_disc_sd}
\begin{align}
\left(\pf_h^{n} - \pf_h^{n-1},\eta_h^\pf\right) + \tau \left(m \nabla \mu_h^{n}, \nabla \eta_h^\pf\right) - \tau \left(R,\eta_h^\pf\right) &= 0, \label{eq:ch1split_sd}\\
\left(\mu_h^{n},\eta_h^\mu\right) - \gamma\left(\ell\left(\nabla\pf_h^{n},\nabla\eta^\mu_h\right)+\frac{1}{\ell}\left(\Psi_c'\!\left(\pf_h^{n}\right)-\Psi_e'\!\left(\pf_h^{n-1}\right),\eta_h^\mu\right)\right) - \left(\mathcal{E}_{\mathrm{e},\pf}^{\mathrm{si}}(\pf_h^n, \strain[u_h^n]; \pf_h^{n-1}, \strain[u_h^{n-1}]),\eta_h^\mu\right)\nonumber\\ - \left(\mathcal{E}_{\mathrm{f},\pf}^{\mathrm{si}}(\pf_h^n, \strain[u_h^n], \theta_h^n; \pf_h^{n-1}, \strain[u_h^{n-1}], \theta_h^{n-1}),\eta_h^\mu\right) &= 0, \label{eq:ch2split_sd}\\
\left(\mathbb{C}(\pf_h^{n-1})\left(\bm\varepsilon\!\left(\bm u_h^{n}\right)-\mathcal{T}\!\left(\pf_h^{n}\right)\right),\bm\varepsilon\!\left(\bm \eta_h^{\bm u}\right)\right) - \left(\alpha(\pf_h^{n}) M(\pf_h^{n-1}) \left(\theta_h^n - \alpha(\pf_h^{n}) \Div \bm u_h^n\right), \Div \bm \eta_h^{\bm u}\right)-\left({\bm f},\bm \eta_h^{\bm u}\right) &= 0, \label{eq:elasticitysplit_sd}\\
\left(\theta_h^{n} - \theta_h^{n-1}, \eta_h^\theta\right) + \tau \left(\kappa \nabla p_h^n, \nabla \eta_h^\theta\right) - \tau \left(S_\mathrm{f}, \eta_h^\theta\right) &= 0, \label{eq:flowsplit_sd}\\
\left(p_h^n, \eta_h^p\right) - \left(M(\pf_h^{n-1}) \left(\theta_h^n - \alpha(\pf_h^{n}) \Div \bm u_h^n\right), \eta_h^p\right) &= 0 \label{eq:pressuresplit_sd}
\end{align}
\end{subequations}
for all $\eta_h^\pf, \eta_h^\mu \in \mathcal{V}_h^\mathrm{ch},$ $\bm \eta_h^{\bm u} \in \mathcal{V}_h^{\bm u}$, and $\eta_h^\theta, \eta_h^{p} \in \mathcal{V}_h^\mathrm{f}$, where we introduce
\begin{align*}
 \mathcal{E}_{\mathrm{e},\pf}^{\mathrm{si}}(\pf_h^n, \strain[u_h^n]; \pf_h^{n-1}, \strain[u_h^{n-1}]) = - \mathcal{T}'\!(\pf_h^n)\!:\!\mathbb{C}(\pf_h^{n-1})\big(\bm\varepsilon\!\left(\bm u_h^n\right)-\mathcal{T}(\pf_h^n)\big) + \frac{1}{2} \big(\strain[u_h^{n-1}] - \mathcal{T}(\pf_h^{n-1})\big)\!:\!\mathbb{C}'\!(\pf_h^{n-1})\big(\strain[u_h^{n-1}] - \mathcal{T}(\pf_h^{n-1})\big)
\end{align*}
and
\begin{align*}
 \mathcal{E}_{\mathrm{f},\pf}^{\mathrm{si}}(\pf_h^n, \strain[u_h^n], \theta_h^n; \pf_h^{n-1}, \strain[u_h^{n-1}], \theta_h^{n-1}) = \frac{M'(\pf_h^{n-1})}{2} \left(\theta_h^{n-1} - \alpha(\pf_h^{n-1}) \Div \bm u_h^{n-1}\right)^2 - \alpha'(\pf_h^n) M(\pf_h^{n-1}) \left(\theta_h^n - \alpha(\pf_h^n) \Div \bm u_h^n\right) \Div \bm u_h^n.
\end{align*}
\begin{remark}
The proposed semi-implicit time discretization is consistent as
\begin{align*}
 \delta_\pf \mathcal{E}_{\mathrm{e}}(\pf, \strain) = \mathcal{E}_{\mathrm{e},\pf}^{\mathrm{si}}(\pf, \strain; \pf, \strain)
\end{align*}
and
\begin{align*}
 \delta_\pf \mathcal{E}_{\mathrm{f}}(\pf, \strain, \theta) = \mathcal{E}_{\mathrm{f},\pf}^{\mathrm{si}}(\pf, \strain, \theta; \pf, \strain, \theta).
\end{align*}
\end{remark}

\begin{lemma}
The solution to the discrete problem~\eqref{eq:model_disc_sd} is equivalent to the solution of the following minimization problem: \\
Given $\pf_h^{n-1} \in \mathcal{V}_h^\mathrm{ch}$, $\bm u_h^{n-1} \in \mathcal{V}_h^{\bm u}$ and $\theta_h^{n-1} \in \mathcal{V}_h^{\mathrm{f}}$, solve
\begin{align}
 \left(\pf_h^n, \bm u_h^n, \theta_h^n\right) = \argmin_{\pf_h \in \mathcal{\bar V}_h^{\mathrm{ch},n}, \bm u_h \in \mathcal{V}_h^{\bm u}, \theta_h \in \mathcal{\bar V}_h^{\mathrm{f},n}} \tilde\calH_\tau^n\!\left(\pf_h, \bm u_h, \theta_h\right)
\label{eq:minim_sd}
\end{align}
with the potential
\begin{align*}
 \tilde\calH_\tau^n\!\left(\pf_h, \bm u_h, \theta_h\right) := &\frac{\lVert \pf_h - \pf_h^{n-1} - \tau R \rVert_{\mathcal{V}_{h,m}^{\mathrm{ch},'}}^2}{2\tau} + \tilde\energy_\mathrm{c}(\pf_h, \strain[u_h], \theta_h; \pf_h^{n-1}) + \frac{\lVert \theta_h - \theta_h^{n-1} - \tau S_\mathrm{f} \rVert_{\mathcal{V}_{h,\kappa}^{\mathrm{f},'}}^2}{2\tau} \\
 &+ \left(\tilde{\mathcal{E}}_e(\pf_h^{n-1}, \strain[u_h^{n-1}], \theta_h^{n-1}),\pf_h\right) - \left({\bm f}, \bm u_h\right)
\end{align*}
with convex energy
\begin{align*}
 \tilde\energy_\mathrm{c}(\pf_h, \strain[u_h], \theta_h; \pf_h^{n-1}) := \tilde\energy_\mathrm{quad}(\pf_h, \strain[u_h], \theta_h; \pf_h^{n-1}) + \frac{\gamma}{\ell} \int_\Omega \Psi_c(\pf_h) \;dx,
\end{align*}
where
\begin{align*}
 \tilde\energy_\mathrm{quad}(\pf_h, \strain[u_h], \theta_h; \pf_h^{n-1}) := &\int_\Omega \frac{\gamma \ell}{2} |\nabla \pf_h|^2 + \frac{1}{2} \big(\bm\varepsilon(\bm u_h) - \mathcal{T}(\pf_h)\big) \!:\!\mathbb{C}\!\left(\pf^{n-1}_h\right) \big(\bm\varepsilon(\bm u_h) - \mathcal{T}(\pf_h)\big) \\ 
 &+ \frac{M\!\left(\pf_h^{n-1}\right)}{2} (\theta_h - \alpha(\pf_h) \Div \bm u_h)^2 \;dx
\end{align*}
and with
\begin{align*}
 \tilde{\mathcal{E}}_e(\pf_h^{n-1}, \strain[u_h^{n-1}], \theta_h^{n-1}) := -& \frac{\gamma}{\ell}\Psi'_e\left(\pf_h^{n-1}\right) + \frac{1}{2}\left(\strain[u_h^{n-1}]-\mathcal{T}\!\left(\pf_h^{n-1}\right)\right):\mathbb{C}'\!\left(\pf_h^{n-1}\right) \left(\strain[u_h^{n-1}]-\mathcal{T}\!\left(\pf_h^{n-1}\right)\right) \\
 &+ \frac{M'\!\left(\pf_h^{n-1}\right)}{2}\left(\theta_h^{n-1}-\alpha\!\left(\pf_h^{n-1}\right)\nabla\cdot\bm u_h^{n-1}\right)^2.
\end{align*}
\label{lem:equivalence_sd}
\end{lemma}
\begin{proof}
Analogous to the proof of Lemma~\ref{lem:equivalence}.
\end{proof}

\begin{remark}[Energy dissipation] \label{rem:disc_ener_diss_sd}
Concerning the evolution of the (state-dependent) free energy 
\begin{align*}
 \energy(\pf, \strain, \theta) = \int_\Omega \gamma \left(\frac{1}{\ell} \Psi(\pf) + \frac{\ell}{2} |\nabla \pf|^2\right) + \frac{1}{2} \big(\bm\varepsilon(\bm u) - \mathcal{T}(\pf)\big) \!:\!\mathbb{C}(\pf) \big(\bm\varepsilon(\bm u) - \mathcal{T}(\pf)\big) + \frac{M(\pf)}{2} (\theta - \alpha(\pf) \Div \bm u)^2 \;dx,
\end{align*}
we carefully follow the same steps as in the proof of Lemma~\ref{lem:disc_ener_diss}, in particular using that following from \eqref{eq:minim_sd} 
\begin{align*}
 \tilde\calH_\tau^n\!\left(\pf_h^n, \bm u_h^n, \theta_h^n\right) \leq \tilde\calH_\tau^n\!\left(\pf_h^{n-1}, \bm u_h^{n-1}, \theta_h^{n-1}\right),
\end{align*}
and after some reformulations we arrive at
\begin{align*}
 &\frac{\lVert \pf_h^n - \pf_h^{n-1}\rVert_{\mathcal{V}_{h,m}^{\mathrm{ch},'}}^2}{2\tau} + \energy(\pf_h^n, \strain[u_h^n], \theta_h^n) + \frac{\lVert \theta_h^n - \theta_h^{n-1}\rVert_{\mathcal{V}_{h,\kappa}^{\mathrm{f},'}}^2}{2\tau} \\
 &\quad+ \left(\frac{1}{2}\left(\strain[u_h^{n-1}]-\mathcal{T}\!\left(\pf_h^{n-1}\right)\right):\mathbb{C}'\!\left(\pf_h^{n-1}\right) \left(\strain[u_h^{n-1}]-\mathcal{T}\!\left(\pf_h^{n-1}\right)\right), \pf_h^n - \pf_h^{n-1}\right) \\
 &\qquad- \int_\Omega \frac{1}{2} \big(\bm\varepsilon(\bm u_h^n) - \mathcal{T}(\pf_h^n)\big) \!:\!\left(\mathbb{C}\!\left(\pf_h^n\right) - \mathbb{C}\!\left(\pf^{n-1}_h\right)\right) \big(\bm\varepsilon(\bm u_h^n) - \mathcal{T}(\pf_h^n)\big) \;dx \\
 &\quad+ \left(\frac{M'\!\left(\pf_h^{n-1}\right)}{2}\left(\theta_h^{n-1}-\alpha\!\left(\pf_h^{n-1}\right)\nabla\cdot\bm u_h^{n-1}\right)^2, \pf_h^n - \pf_h^{n-1}\right) \\
 &\qquad- \int_\Omega \frac{M\!\left(\pf_h^n\right) - M\!\left(\pf_h^{n-1}\right)}{2} (\theta_h^n - \alpha(\pf_h^n) \Div \bm u_h^n)^2 \;dx \\
 &\leq \energy\!\left(\pf_h^{n-1}, \strain[u_h^{n-1}], \theta_h^{n-1}\right).
\end{align*}
We note that even by reducing the time-step size $\tau$, we do not have control over the strain-related terms. This may lead to a non-dissipative behavior which we further investigate in our numerical experiments in Section~\ref{sec:experiments}.
\end{remark}

\subsubsection{Alternating minimization} \label{sec:alter_min_sd}

For the $i$-th iteration of the alternating minimization scheme we minimize the potential $\tilde\calH_\tau^n$ as
\begin{align*}
 \pf_h^{n,i} &= \argmin_{\pf_h \in \mathcal{\bar V}_h^{\mathrm{ch},n}} \tilde\calH_\tau^n\!\left(\pf_h, \bm u_h^{n,i-1}, \theta_h^{n,i-1}\right) \\
 \left(\bm u_h^{n,i}, \theta_h^{n,i}\right) &= \argmin_{\bm u_h \in \mathcal{V}_h^{\bm u}, \theta_h \in \mathcal{\bar V}_h^{\mathrm{f},n}} \tilde\calH_\tau^n\!\left(\pf_h^{n,i}, \bm u_h, \theta_h\right)
\end{align*}
and based on the variational system of equations we obtain: \\
Given $\pf_h^{n-1} \in \mathcal{V}_h^{\mathrm{ch}}$, $\bm u_h^{n-1}, \bm u_h^{n,i-1} \in \mathcal{V}_h^{\bm u}$, and $\theta_h^{n-1}, \theta_h^{n,i-1} \in \mathcal{V}_h^{\mathrm{f}}$, find $\pf_h^{n,i}, \mu_h^{n,i} \in \mathcal{V}_h^{\mathrm{ch}}$, $\bm u_h^{n,i} \in \mathcal{V}_h^{\bm u}$, and $\theta_h^{n,i}, p_h^{n,i} \in \mathcal{V}_h^{\mathrm{f}}$ such that
\begin{subequations}
\label{eq:model_disc_alter_sd}
\begin{align}
\left(\pf_h^{n,i} - \pf_h^{n-1},\eta_h^\pf\right) + \left( m \nabla \mu_h^{n,i}, \nabla \eta_h^\pf\right)- \tau \left(R,\eta_h^\pf\right) &= 0, \label{eq:ch1split_alter_sd}\\
\left(\mu_h^{n,i},\eta_h^\mu\right) - \gamma\left(\ell\left(\nabla\pf_h^{n,i},\nabla\eta^\mu_h\right)+\frac{1}{\ell}\left(\Psi_c'\!\left(\pf_h^{n,i}\right)-\Psi_e'\!\left(\pf_h^{n-1}\right),\eta_h^\mu\right)\right)\nonumber\\ - \left(\mathcal{E}_{\mathrm{e},\pf}^{\mathrm{si}}(\pf_h^{n,i}, \strain[u_h^{n,i-1}]; \pf_h^{n-1}, \strain[u_h^{n-1}]),\eta_h^\mu\right) - \left(\mathcal{E}_{\mathrm{f},\pf}^{\mathrm{si}}(\pf_h^{n,i}, \strain[u_h^{n,i-1}], \theta_h^{n,i-1}; \pf_h^{n-1}, \strain[u_h^{n-1}], \theta_h^{n-1}),\eta_h^\mu\right) &= 0, \label{eq:ch2split_alter_sd}\\
\left(\mathbb{C}(\pf_h^{n-1}) \left(\bm\varepsilon\!\left(\bm u_h^{n,i}\right)-\mathcal{T}\!\left(\pf_h^{n,i}\right)\right),\bm\varepsilon\!\left(\bm \eta_h^{\bm u}\right)\right) - \left(\alpha(\pf_h^{n,i}) M(\pf_h^{n-1}) \left(\theta_h^{n,i} - \alpha(\pf_h^{n,i}) \Div \bm u_h^{n,i}\right), \Div \bm \eta_h^{\bm u}\right)-\left({\bm f},\bm \eta_h^{\bm u}\right) &= 0, \label{eq:elasticitysplit_alter_sd}\\
\left(\theta_h^{n,i} - \theta_h^{n-1}, \eta_h^\theta\right) + \tau \left(\kappa \nabla p_h^{n,i}, \nabla \eta_h^\theta\right) - \tau \left(S_\mathrm{f}, \eta_h^\theta\right) &= 0, \label{eq:flowsplit_alter_sd}\\
\left(p_h^{n,i}, \eta_h^p\right) - \left(M(\pf_h^{n-1})\left(\theta_h^{n,i} - \alpha(\pf_h^{n,i}) \Div \bm u_h^{n,i}\right), \eta_h^p\right) &= 0 \label{eq:pressuresplit_alter_sd}
\end{align}
\end{subequations}
for all $\eta_h^\pf, \eta_h^\mu \in \mathcal{V}_h^\mathrm{ch},$ $\bm \eta_h^{\bm u} \in \mathcal{V}_h^{\bm u}$, and $\eta_h^\theta, \eta_h^p \in \mathcal{V}_h^\mathrm{f}$.

\begin{theorem}
 The alternating minimization scheme~\eqref{eq:model_disc_alter_sd} converges linearly in the sense that
\begin{align}
 \tilde\calH_\tau^n\!\left(\pf_h^{n,i},\bm u_h^{n,i},\theta_h^{n,i}\right) - \tilde\calH_\tau^n\!\left(\pf_h^{n},\bm u_h^{n},\theta_h^{n}\right) \leq \left(1 - \frac{\beta_\mathrm{ch}}{L_\mathrm{ch}}\right) \left(1 - \frac{\beta_\mathrm{b}}{L_\mathrm{b}}\right) \Big( \tilde\calH_\tau^n\!\left(\pf_h^{n,i-1},\bm u_h^{n,i-1},\theta_h^{n,i-1}\right) - \tilde\calH_\tau^n\!\left(\pf_h^{n},\bm u_h^{n},\theta_h^{n}\right) \Big),
 \label{eq:conv_result_sd}
\end{align}
where $\beta_\mathrm{ch} = 1 - \left(\frac{h^2}{\tau C_\mathrm{inv}^2 C_\mathbb{C} C_\mathcal{T}^2} + \frac{\gamma \ell}{C_\Omega^2 C_\mathbb{C} C_\mathcal{T}^2} + 1\right)^{-1}$, $\beta_\mathrm{b} = \min\left\{1 - \left(\frac{h^2}{\tau C_\mathrm{inv}^2 C_\mathbb{C} C_\mathcal{T}^2} + \frac{\gamma \ell}{C_\Omega^2 C_\mathbb{C} C_\mathcal{T}^2} + 1\right)^{-1}, 1 - \sqrt{\left(\frac{1}{C_\alpha^2} \left[\frac{h^2}{\tau C_M C_\mathrm{inv}^2} + 1\right]\right)^{-1}}\right\}$, $L_\mathrm{ch} = 1 + \frac{\gamma}{\ell} L_{\Psi_\mathrm{c}'} \left(\frac{1}{\tau} \frac{h^2}{C_\mathrm{inv}^2} + \gamma \ell \frac{1}{C_\Omega^2} + c_\mathbb{C} c_\mathcal{T}^2\right)^{-1}$, and $L_\mathrm{b} = 2$.
\label{thm:conv_result_sd}
\end{theorem}
\begin{proof}
The proof is similar to the one of Theorem~\ref{thm:conv_result}. We define the (semi-)norms
\begin{align*}
 \lVert (\pf_h, \bm u_h, \theta_h) \rVert^2 &:= \frac{\lVert \pf_h \rVert_{\mathcal{V}_{h,m}^{\mathrm{ch},'}}^2}{\tau} + \gamma \ell \lVert \nabla\pf_h \rVert_{\vecSpL}^2 + \left(\mathbb{C}(\pf_h^{n-1})\left(\bm\varepsilon\!\left(\bm u_h\right)-\mathcal{T}\!\left(\pf_h\right)\right),\bm\varepsilon\!\left(\bm u_h\right)-\mathcal{T}\!\left(\pf_h\right)\right) \\ 
 &\qquad+ \frac{\lVert \theta_h \rVert_{\mathcal{V}_{h,\kappa}^{\mathrm{f},'}}^2}{\tau} + \left(M(\pf_h^{n-1}) \left(\theta_h - \alpha(\pf_h) \Div \bm u_h\right), \theta_h - \alpha(\pf_h) \Div \bm u_h\right), \\
 \lVert \pf_h \rVert_\mathrm{ch}^2 &:= \frac{\lVert \pf_h \rVert_{\mathcal{V}_{h,m}^{\mathrm{ch},'}}^2}{\tau} + \gamma \ell \lVert \nabla\pf_h \rVert_{\vecSpL}^2 + \left(\mathbb{C}(\pf_h^{n-1})\mathcal{T}\!\left(\pf_h\right), \mathcal{T}\!\left(\pf_h\right)\right), \\
 \lVert (\bm u_h, \theta_h) \rVert_\mathrm{b}^2 &:= \left(\mathbb{C}(\pf_h^{n-1})\bm\varepsilon\!\left(\bm u_h\right), \bm\varepsilon\!\left(\bm u_h\right)\right) + \frac{\lVert \theta_h \rVert_{\mathcal{V}_{h,\kappa}^{\mathrm{f},'}}^2}{\tau} + \left(M(\pf_h^{n-1}) \theta_h, \theta_h\right) + \left(M(\pf_h^{n-1}) \Div \bm u_h, \Div \bm u_h\right)
\end{align*}
for $\pf_h \in \mathcal{V}_{h,0}^\mathrm{ch}$, $\bm u_h \in \mathcal{V}_h^{\bm u}$ and $\theta_h \in \mathcal{V}_{h,0}^{\mathrm{f}}$, and accordingly arrive at the bound
\begin{align*}
 \lVert (\pf_h, \bm u_h, \theta_h) \rVert^2 &\geq
 \left(\frac{1}{\tau} - \frac{1}{\lambda} c_1 \frac{C_\mathrm{inv}^2}{h^2} C_\mathbb{C} C_\mathcal{T}^2\right) \lVert \pf_h \rVert_{\mathcal{V}_{h,m}^{\mathrm{ch},'}}^2 + \left(\gamma \ell - \frac{1}{\lambda} c_2 C_\Omega^2 C_\mathbb{C} C_\mathcal{T}^2\right) \lVert \nabla \pf_h \rVert_{\vecSpL} \\
 &\quad + \left(1 - \frac{c_3}{\lambda}\right) \left(\mathbb{C}(\pf_h^{n-1})\mathcal{T}\!\left(\pf_h\right), \mathcal{T}\!\left(\pf_h\right)\right) + \left(1 - \lambda\right) \left(\mathbb{C}(\pf_h^{n-1})\bm\varepsilon\!\left(\bm u_h\right), \bm\varepsilon\!\left(\bm u_h\right)\right) \\
 &\qquad + \frac{\lVert \theta_h \rVert_{\mathcal{V}_{h,\kappa}^{\mathrm{f},'}}^2}{\tau} + \left(M(\pf_h^{n-1}) \left(\theta_h - \alpha(\pf_h) \Div \bm u_h\right), \theta_h - \alpha(\pf_h) \Div \bm u_h\right).
\end{align*}
The relation \eqref{eq:rel_norm_ch} then follows with the same constant as in Theorem~\ref{thm:conv_result}. \\
We further have
\begin{align*}
 &\left(M(\pf_h^{n-1}) \left(\theta_h - \alpha(\pf_h) \Div \bm u_h\right), \theta_h - \alpha(\pf_h) \Div \bm u_h\right) \\
 &\qquad= \left(M(\pf_h^{n-1}) \theta_h, \theta_h\right) - 2 \left(M(\pf_h^{n-1}) \theta_h, \alpha(\pf_h) \Div \bm u_h\right) + \left(M(\pf_h^{n-1}) \alpha(\pf_h) \Div \bm u_h, \alpha(\pf_h) \Div \bm u_h\right)
\end{align*}
and by Young's inequality, the boundedness of $\alpha(\cdot)$ and \eqref{eq:inv_eq_dual} we get
\begin{align*}
 2 \left(M(\pf_h^{n-1}) \theta_h, \alpha(\pf_h) \Div \bm u_h\right) &\leq \frac{1}{\zeta} (\tilde{c}_1 + \tilde{c}_2) \left(M(\pf_h^{n-1}) \theta_h, \theta_h\right) + \zeta \left(M(\pf_h^{n-1}) \alpha(\pf_h) \Div \bm u_h, \alpha(\pf_h) \Div \bm u_h\right) \\
 &\leq \frac{1}{\zeta} \tilde{c}_1 C_M \|\theta_h\|_{L^2{\Omega}}^2 + \frac{1}{\zeta} \tilde{c}_2 \left(M(\pf_h^{n-1}) \theta_h, \theta_h\right) + \zeta C_\alpha^2 \left(M(\pf_h^{n-1}) \Div \bm u_h, \Div \bm u_h\right) \\
 &\leq \frac{1}{\zeta} \tilde{c}_1 C_M \frac{C_\mathrm{inv}^2}{h^2} \lVert \theta_h \rVert_{\mathcal{V}_{h,\kappa}^{\mathrm{f},'}}^2 + \frac{1}{\zeta} \tilde{c}_2 \left(M(\pf_h^{n-1}) \theta_h, \theta_h\right) + \zeta C_\alpha^2 \left(M(\pf_h^{n-1}) \Div \bm u_h, \Div \bm u_h\right)
\end{align*}
with coefficients $0 \leq \tilde{c}_k \leq 1$ for $k \in \{1,2\}$, $\tilde{c}_1 + \tilde{c}_2= 1$, and $\zeta > 0$ to be chosen later, yielding
\begin{align*}
 \lVert (\pf_h, \bm u_h, \theta_h) \rVert^2 &\geq
 \left(\frac{1}{\tau} - \frac{1}{\lambda} c_1 \frac{C_\mathrm{inv}^2}{h^2} C_\mathbb{C} C_\mathcal{T}^2\right) \lVert \pf_h \rVert_{\mathcal{V}_{h,m}^{\mathrm{ch},'}}^2 + \left(\gamma \ell - \frac{1}{\lambda} c_2 C_\Omega^2 C_\mathbb{C} C_\mathcal{T}^2\right) \lVert \nabla \pf_h \rVert_{\vecSpL} \\
 &\quad + \left(1 - \frac{c_3}{\lambda}\right) \left(\mathbb{C}(\pf_h^{n-1})\mathcal{T}\!\left(\pf_h\right), \mathcal{T}\!\left(\pf_h\right)\right) + \left(1 - \lambda\right) \left(\mathbb{C}(\pf_h^{n-1})\bm\varepsilon\!\left(\bm u_h\right), \bm\varepsilon\!\left(\bm u_h\right)\right) \\
 &\qquad + \left(\frac{1}{\tau} - \frac{1}{\zeta} \tilde{c}_1 C_M \frac{C_\mathrm{inv}^2}{h^2}\right) \lVert \theta_h \rVert_{\mathcal{V}_{h,\kappa}^{\mathrm{f},'}}^2 \\
 &\qquad \quad+ \left(1 - \frac{\tilde{c}_2}{\zeta}\right) \left(M(\pf_h^{n-1}) \theta_h, \theta_h\right) + \left(1 - \zeta C_\alpha^2\right) \left(M(\pf_h^{n-1}) \Div \bm u_h, \Div \bm u_h\right).
\end{align*}
Choosing
\begin{align*}
 c_1 = \lambda \frac{h^2}{\tau C_\mathrm{inv}^2 C_\mathbb{C} C_\mathcal{T}^2}, \quad c_2 = \lambda \frac{\gamma \ell}{C_\Omega^2 C_\mathbb{C} C_\mathcal{T}^2}, \quad c_3 = \lambda, \quad \lambda = \left(\frac{h^2}{\tau C_\mathrm{inv}^2 C_\mathbb{C} C_\mathcal{T}^2} + \frac{\gamma \ell}{C_\Omega^2 C_\mathbb{C} C_\mathcal{T}^2} + 1\right)^{-1} 
\end{align*}
as in the proof of Theorem~\ref{thm:conv_result}, we obtain
\begin{align*}
 \lVert (\pf_h, \bm u_h, \theta_h) \rVert^2 &\geq \left(1 - \lambda\right) \left(\mathbb{C}(\pf_h^{n-1})\bm\varepsilon\!\left(\bm u_h\right), \bm\varepsilon\!\left(\bm u_h\right)\right) + \left(\frac{1}{\tau} - \frac{1}{\zeta} \tilde{c}_1 C_M \frac{C_\mathrm{inv}^2}{h^2}\right) \lVert \theta_h \rVert_{\mathcal{V}_{h,\kappa}^{\mathrm{f},'}}^2 \\
 &\qquad \quad+ \left(1 - \frac{\tilde{c}_2}{\zeta}\right) \left(M(\pf_h^{n-1}) \theta_h, \theta_h\right) + \left(1 - \zeta C_\alpha^2\right) \left(M(\pf_h^{n-1}) \Div \bm u_h, \Div \bm u_h\right).
\end{align*}
By introducing
\begin{align*}
 \tilde{c}_1 = (1 - \tilde \beta_\mathrm{b}) \zeta \frac{h^2}{\tau C_M C_\mathrm{inv}^2}, \quad \tilde{c}_2 = (1 - \tilde \beta_\mathrm{b}) \zeta
\end{align*}
and
\begin{align*}
 \zeta = \frac{1}{C_\alpha^2} (1 - \tilde \beta_\mathrm{b}) > 0 \quad \Rightarrow \quad \tilde \beta_\mathrm{b} < 1
\end{align*}
together with the relation $\tilde c_1 + \tilde c_2 = 1$, we derive
\begin{align*}
 (1 - \tilde \beta_\mathrm{b})^2 \frac{1}{C_\alpha^2} \left[\frac{h^2}{\tau C_M C_\mathrm{inv}^2} + 1\right] = 1 \quad \Leftrightarrow \quad \tilde \beta_\mathrm{b} = 1 - \sqrt{\left(\frac{1}{C_\alpha^2} \left[\frac{h^2}{\tau C_M C_\mathrm{inv}^2} + 1\right]\right)^{-1}}
\end{align*}
with $\tilde \beta_\mathrm{b} > 0$ as $C_\alpha \leq 1$, and hence
\begin{align*}
 \lVert (\pf_h, \bm u_h, \theta_h) \rVert^2 &\geq \left(1 - \lambda\right) \left(\mathbb{C}(\pf_h^{n-1})\bm\varepsilon\!\left(\bm u_h\right), \bm\varepsilon\!\left(\bm u_h\right)\right) + \tilde \beta_\mathrm{b} \left(\frac{1}{\tau} \lVert \theta_h \rVert_{\mathcal{V}_{h,\kappa}^{\mathrm{f},'}}^2 + \left(M(\pf_h^{n-1}) \theta_h, \theta_h\right) +  \left(M(\pf_h^{n-1}) \Div \bm u_h, \Div \bm u_h\right)\right) \\
 &\geq \beta_\mathrm{b} \lVert (\bm u_h, \theta_h) \rVert_\mathrm{b}^2
\end{align*}
with constant $\beta_\mathrm{b} = \min\{1 - \lambda, \tilde \beta_\mathrm{b}\}$, i.e., \eqref{eq:rel_norm_b}.

The potential related bounds \eqref{eq:H_conv} and \eqref{eq:Lip_cont_ch} follow similarly as in the proof of Theorem~\ref{thm:conv_result} with corresponding constants. \\
Lastly, we employ the definition of $\tilde\energy_\mathrm{quad}$ along with the boundedness of $\alpha(\cdot)$ and Young's inequality to obtain
\begin{align*}
 &\Big(\delta_{(\bm u,\theta)}\tilde\calH\!\left(\pf_h, \bm u_h^1, \theta_h^1\right) - \delta_{(\bm u,\theta)}\tilde\calH\!\left(\pf_h, \bm u_h^2, \theta_h^2\right), \left(\bm u_h^1 - \bm u_h^2, \theta_h^1 - \theta_h^2\right)\Big) \\
 &= \Big(\delta_{(\strain,\theta)} \tilde\energy_\mathrm{quad}(\pf_h^1, \strain[u_h^1], \theta_h^1; \pf_h^{n-1}) - \delta_{(\strain,\theta)} \tilde\energy_\mathrm{quad}(\pf_h^2, \strain[u_h^2], \theta_h^2; \pf_h^{n-1}), \left(\strain[u_h^1 - \bm u_h^2], \theta_h^1 - \theta_h^2\right)\Big) + \left(\frac{\theta_h^1 - \theta_h^2}{\tau}, \theta_h^1 - \theta_h^2\right)_{\mathcal{V}_{h,\kappa}^{\mathrm{f},'}} \\
 &= 2 \tilde\energy_\mathrm{quad}(0, \strain[u_h^1 - \bm u_h^2], \theta_h^1 - \theta_h^2; \pf_h^{n-1}) + \frac{\lVert \theta_h^1 - \theta_h^2 \rVert_{\mathcal{V}_{h,\kappa}^{\mathrm{f},'}}^2}{\tau} \\
 &= \left(\mathbb{C}(\pf_h^{n-1})\bm\varepsilon\!\left(\bm u_h^1 - \bm u_h^2\right), \bm\varepsilon\!\left(\bm u_h^1 - \bm u_h^2\right)\right) + \frac{\lVert \theta_h^1 - \theta_h^2 \rVert_{\mathcal{V}_{h,\kappa}^{\mathrm{f},'}}^2}{\tau} \\
 &\qquad+ \left(M(\pf_h^{n-1}) \left([\theta_h^1 - \theta_h^2] - \alpha(0) \Div [\bm u_h^1 - \bm u_h^2]\right), [\theta_h^1-\theta_h^2] - \alpha(0) \Div [\bm u_h^1 - \bm u_h^2]\right) \\
 &\leq \left(\mathbb{C}(\pf_h^{n-1})\bm\varepsilon\!\left(\bm u_h^1 - \bm u_h^2\right), \bm\varepsilon\!\left(\bm u_h^1 - \bm u_h^2\right)\right) + \frac{\lVert \theta_h^1 - \theta_h^2 \rVert_{\mathcal{V}_{h,\kappa}^{\mathrm{f},'}}^2}{\tau} \\
 &\qquad+ 2 \left(M(\pf_h^{n-1}) (\theta_h^1 - \theta_h^2), \theta_h^1-\theta_h^2\right) + 2 C_\alpha^2 \left(M(\pf_h^{n-1}) \Div (\bm u_h^1 - \bm u_h^2), \Div (\bm u_h^1 - \bm u_h^2)\right) \\
 &\leq L_\mathrm{b} \lVert (\bm u_h^1 - \bm u_h^2, \theta_h^1 - \theta_h^2) \rVert_\mathrm{b}^2
\end{align*}
with constant $L_\mathrm{b} = \max\{2, 2C_\alpha^2\} = 2$ for all $\pf_h \in \mathcal{\bar V}_h^{\mathrm{ch},n}$, $\bm u_h^1, \bm u_h^2 \in \mathcal{V}_h^{\bm u}$, and $\theta_h^1, \theta_h^2 \in \mathcal{\bar V}_h^{\mathrm{f},n}$, and thus the remaining bound \eqref{eq:Lip_cont_b}.

The desired convergence result \eqref{eq:conv_result_sd} then follows from Lemma~\ref{lem:aux_conv_result}.
\end{proof}

\section{Numerical experiments} \label{sec:experiments}

The numerical experiments are tailored to assess the robustness and efficacy of the proposed semi-implicit time discretization \eqref{eq:model_disc_sd} and the subsequently introduced alternating minimization method \eqref{eq:model_disc_alter_sd}. 
To do so, we employ a comparative study for three different solution schemes; (fully) monolithic \eqref{eq:model_disc_sd}, alternating minimization between the Cahn-Hilliard and the monolithic solution of the Biot subsystem \eqref{eq:model_disc_alter_sd}, and a three way split, where the Cahn-Hilliard \eqref{eq:ch1split_sd}-\eqref{eq:ch2split_sd}, the elasticity \eqref{eq:elasticitysplit_sd} and the flow \eqref{eq:flowsplit_sd}-\eqref{eq:pressuresplit_sd} subsystems are all solved sequentially. We then apply these based on a ''fully'' implicit time discretization and the proposed semi-implicit one, leading to a total of six different solution strategies. We note that the ''fully'' implicit time discretization is in fact not fully implicit as the classical Eyre split \cite{Eyre1998} is used to split the double-well potential.

For the spatial discretization, we triangulate the domain $\Omega$ and employ first order conforming finite elements for all variables $\pf$, $\mu$, $\theta$, $p$, and $\bm u$, the latter naturally being vectorial.

We introduce the interpolation function
\begin{align}
 \pi(\pf) =
 \begin{cases}
  0 &\textrm{for } \pf < -1, \\
  \frac{1}{4} \left(-\pf^3 + 3\pf + 2\right) &\textrm{for } \pf \in \left[-1,1\right], \\
  1 &\textrm{for } \pf > 1,
 \end{cases}
\label{eq:interpolation_function}
\end{align}
in order to write the nonlinear state-dependent material parameters as
\begin{align*}
 \zeta(\pf) = \zeta_{-1} + \pi(\pf)\!\left(\zeta_1-\zeta_{-1}\right)
\end{align*}
with $\zeta \in \left\{M, \alpha, \mathbb{C}\right\}$.
For the eigenstrain we choose $\mathcal{T}(\pf) = \xi \pf \bm I$ \cite{Areias2016} accounting for swelling effects with constant swelling parameter $\xi > 0$.
The specific material, discretization and linearization parameters are presented in Table~\ref{tab:simulation_parameters}, with the stiffness tensors given in Voigt notation as 
\begin{align}
 \mathbb{C}_{-1} = 
 \begin{pmatrix}
  100 & 20  & 0\\
  20  & 100 & 0\\
  0   & 0   & 100
 \end{pmatrix}, \quad \mathrm{and}\quad 
 \mathbb{C}_{1} = 
 \begin{pmatrix}
  1   & 0.1 & 0\\
  0.1 & 1   & 0\\
  0   & 0   & 1
 \end{pmatrix}.
\label{eq:stiffness_tensors}
\end{align}

We use an absolute incremental stopping criterion of the form 
\begin{equation*}
    \left\|\pf_h^{n,i}-\pf_h^{n,i-1}\right\|_{L^2(\Omega)}^2+\left\|\bm u_h^{n,i}-\bm u_h^{n,i-1}\right\|_{L^2(\Omega)}^2 +\left\|p_h^{n,i}-p_h^{n,i-1}\right\|_{L^2(\Omega)}^2<\mathrm{tol}.
\end{equation*}

\begin{table}[h]
\centering
\begin{tabular}{l|c|c|c}
Parameter name                    & Symbol                         & Value         &  Unit \\
\hline
Surface tension                   & $\gamma$                       & $1$           & $\left[\frac{F}{L}\right]$\\
Regularization parameter          & $\ell$                         & $0.025$       & $[L]$ \\
Mobility                          & $m$                            & $1.0$         & $\left[\frac{L^4}{FT}\right]$  \\
Permeability                      & $\kappa$                       & $0.25$         & $\left[\frac{L^4}{FT}\right]$ \\
Swelling parameter                & $\xi$                          & $0.5$         & --\\
Stiffness tensors                 & 
$\mathbb{C}_{-1}$, $\mathbb{C}_1$ & -             & $\left[\frac{F}{L^2}\right]$\\
Compressibilities                 & 
$M_{-1}$, $M_1$                   & $1.0$, $0.1$  & $\left[\frac{F}{L^2}\right]$\\
Biot-Willis coupling coefficients & 
$\alpha_{-1}$, $\alpha_1$         & $1.0$, $0.1$  & -- \\
Time-step size                    & $\tau$                         & $1$e$-3$      & $[T]$\\
Mesh diameter                     & $h$                            & $\sqrt{2}/64$ & $[L]$\\
Max nonlinear iteration           & max\textunderscore iter        & $100$         & -- \\
Tolerance                         & tol                            & $1$e$-6$      & -- \\
\end{tabular}
\caption{Default simulation parameters. The stiffness tensors are given in \eqref{eq:stiffness_tensors} and the interpolation function in \eqref{eq:interpolation_function}. Here, $L$, $T$ and $F$ denote the units of length, time and force respectively.}
\label{tab:simulation_parameters}
\end{table}

For the modified double-well potential and the related convex-concave split, cf.~Remark~\ref{rem:double-well}, we choose $\beta = \frac{3}{2}$.

Note that the Cahn-Hilliard subsystem for the alternating minimization scheme \eqref{eq:model_disc_alter_sd} is still nonlinear due to the potential term $\Psi_c'\!\left(\pf_h^{n,i}\right)$. Here, we employ Newton's method to solve it, which is known to converge in this setting \cite{GuillenGonzalez2014}.

For the implementation we used FEniCSx \cite{Baratta2023, Scroggs2022, Basix2022, Alnaes2014} version 0.9, the code is publicly available \cite{Storvik2026}, and all computations are carried out on a MacBook Pro with the Apple M3 chip.

\subsection{Model problem}
The initial phase field takes value $-1$ in the left half of the domain and value $1$ in the right half. The other fields are initialized with value zero in the entire domain. We apply homogeneous Neumann boundary conditions to the Cahn-Hilliard and flow subproblems, and homogeneous Dirichlet conditions to the elasticity subproblem. The source terms $R$, $S_f$ and $\bm f$ are all set to zero. In Figure~\ref{fig:solution_model_problem} we show simulation results for the evolution of the phase field $\pf$ given this initial setup and the parameters from Table~\ref{tab:simulation_parameters}.

We note that only the outermost iterations are counted in the presentation, i.e.:
\begin{itemize}
    \item For the monolithic solvers we count the number of Newton iterations.
    \item For the split between the Cahn-Hilliard and the Biot subsystem we count the number of coupling iterations, i.e., solving each of the subsystems once counts as one iteration in total. The Newton iterations used to solve the Cahn-Hilliard subsystem are not counted.
    \item For the three-way, fixed-stress type split solving each of the three subsystems once is counted as one iteration in total. Again, the Newton iterations used to solve the Cahn-Hilliard system are not counted.  
\end{itemize}

\begin{figure}
    \centering
    \begin{subfigure}{0.24\linewidth}
            \includegraphics[width=\linewidth]{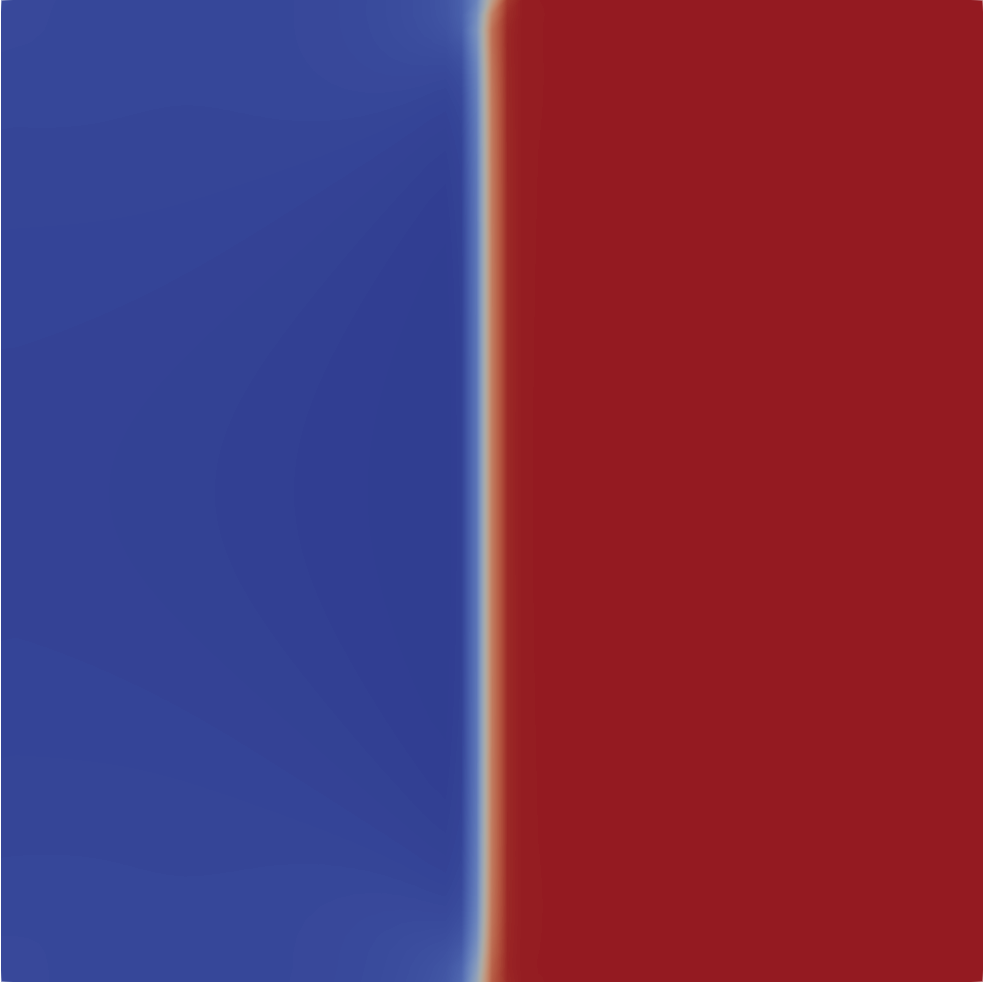}
            \caption{$t=0$}
    \end{subfigure}
    \begin{subfigure}{0.24\linewidth}
            \includegraphics[width=\linewidth]{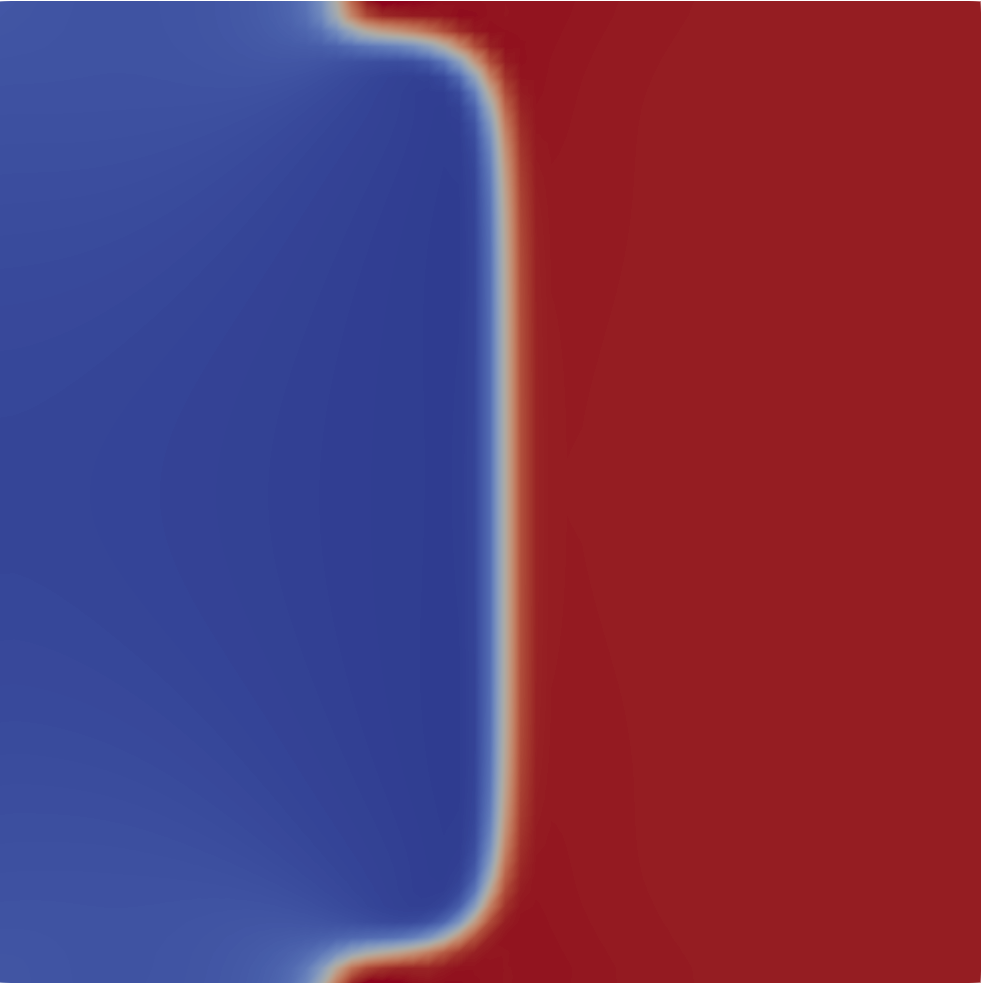}
            \caption{$t=0.01$s}
    \end{subfigure}
    \begin{subfigure}{0.24\linewidth}
            \includegraphics[width=\linewidth]{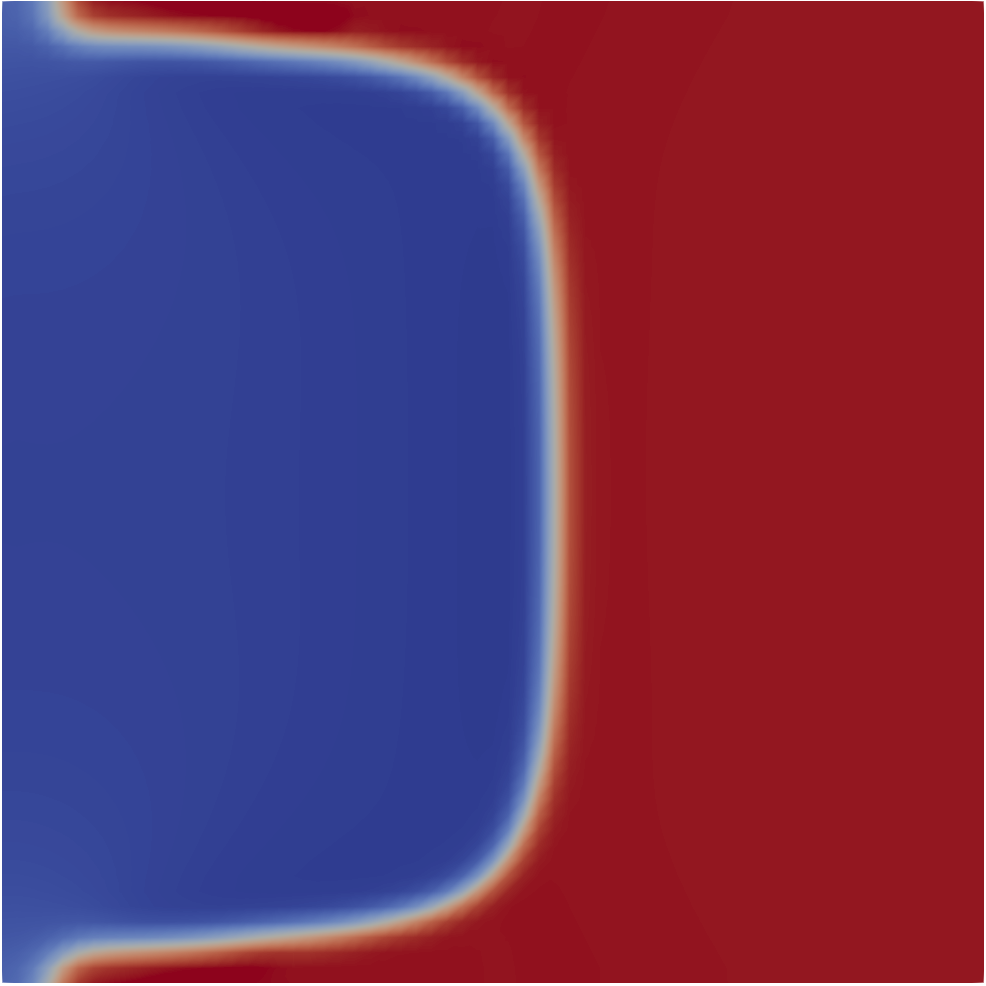}
            \caption{$t=0.05$s}
    \end{subfigure}
    \begin{subfigure}{0.24\linewidth}
            \includegraphics[width=\linewidth]{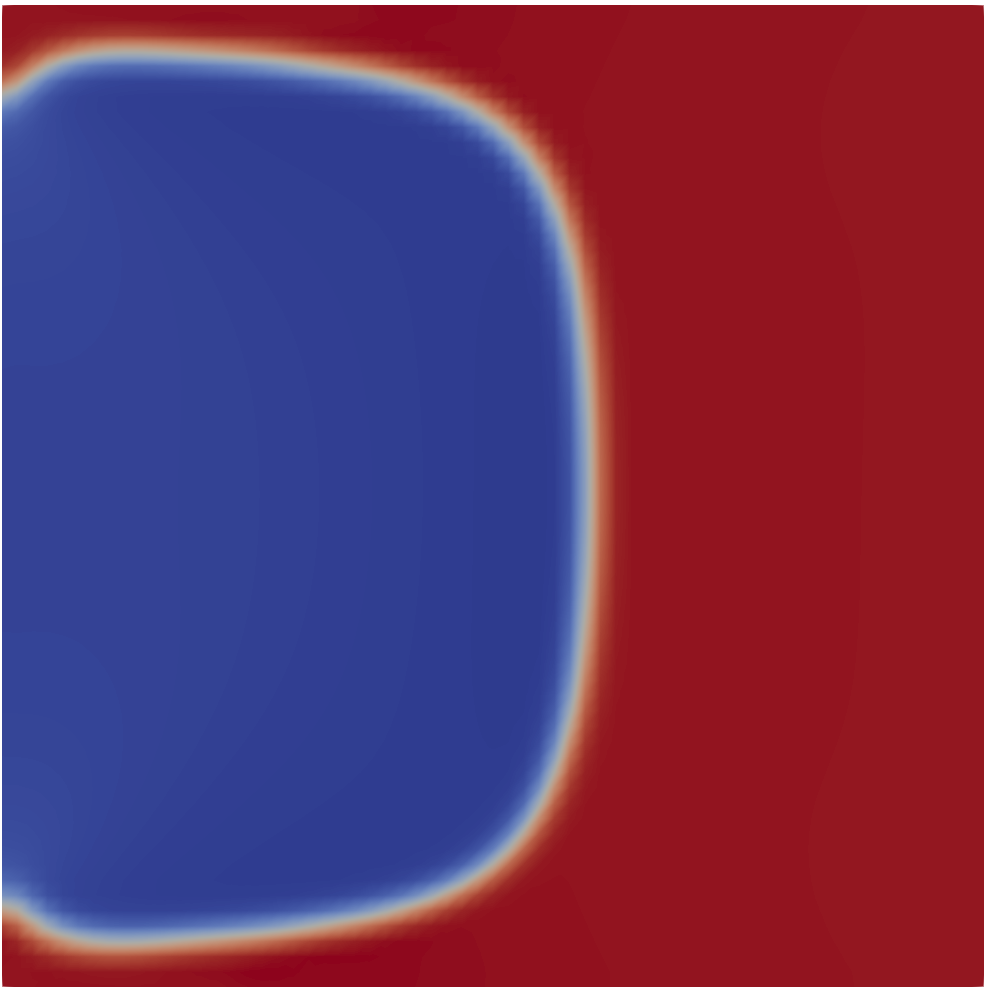}
            \caption{$t=0.1$s}
    \end{subfigure}
    \begin{subfigure}{\linewidth}
                \includegraphics[width=\linewidth]{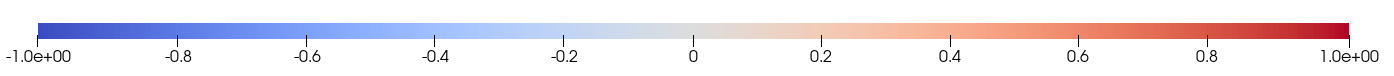}
    \end{subfigure}
    \caption{Evolution of the phase field $\varphi$ for the model problem. Presented above are four snap-shots at different times.}
    \label{fig:solution_model_problem}
\end{figure}

\subsubsection{Dependence on surface tension $\gamma$}
In Figure~\ref{fig:gamma} we present the results of a study with varying values of the surface tension parameter $\gamma$. The remaining material and simulation parameters are found in Table~\ref{tab:simulation_parameters}, specifically $\xi = 0.5$. 
The study has several interesting implications. Firstly, increasing the value of $\gamma$ results in relatively simpler simulations in terms of the convergence speed (fewer iterations and lower computational time), which is reasonable as it effectively reduces the coupling strength between the Cahn-Hilliard and Biot subsystems. Secondly, the semi-implicit time discretization is more suited for linearization methods, at least the decoupling ones, than the implicit discretization. This is also the point of the temporal discretization as the discrete nonlinear systems are designed to be convex in this case. Thirdly, we see that although the numbers of iterations are far lower for the monolithic solvers, the decoupling solvers are more competitive in terms of computational time. We expect this to be due to the ease of solving each of the nonlinear and linear systems related to the separate subproblems, rather than the full one. Note that we exclusively used the default (non-)linear solver configurations from PETSc \cite{petsc, Dalcin2011} through DOLFINx \cite{Baratta2023} and expect that more efficient solver setups along with tailor-made preconditioners could be established for the monolithic problem, as well as for each subproblem. 
\pgfplotstableread{
0 0
1 0
80 0
}\datatableEmpty

\pgfplotstableread{
0.25	342
0.5	310
1	264
2	238
4	205
}\datatableGammaIterMonoSemi

\pgfplotstableread{
0.25 253
0.5	204
1	203
2	279
4	281
}\datatableGammaIterMono

\pgfplotstableread{
0.25 1033
0.5	746
1	571
2	468
4	401
}\datatableGammaIterSplitSemi

\pgfplotstableread{
0.25	1505
0.5	1275
1	994
2	926
4	687
}\datatableGammaIterSplit

\pgfplotstableread{
0.25 1026
0.5	740
1	564
2	461
4	401
}\datatableGammaIterFSSemi

\pgfplotstableread{
0.25 1471
0.5	1263
1	985
2	917
4	677
}\datatableGammaIterFS

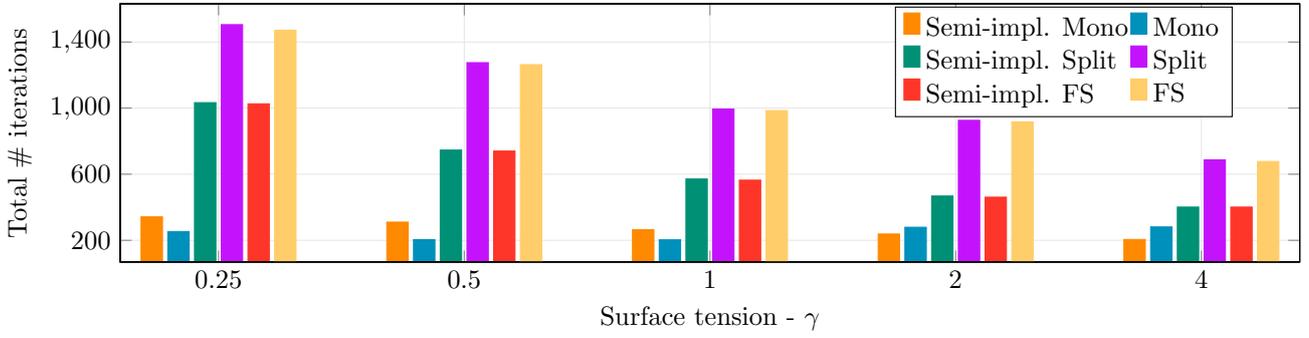
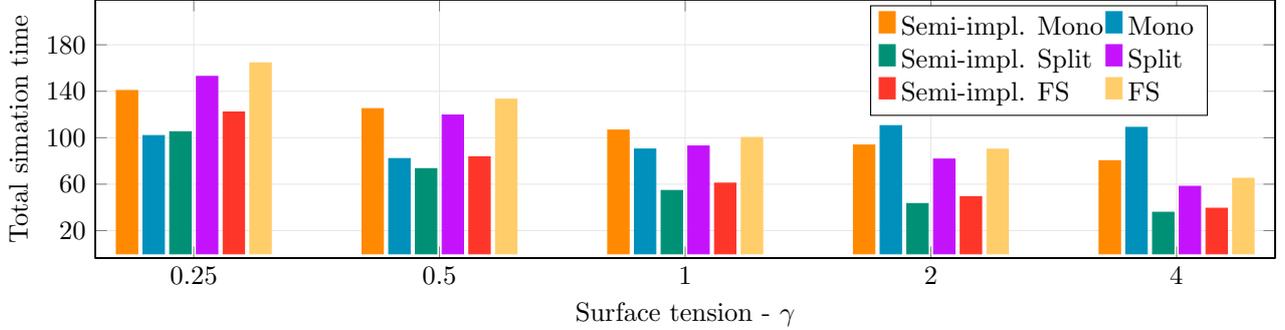
\begin{figure}
\centering
\hspace{-0.78cm}
\begin{subfigure}{0.9\linewidth}
  \begin{tikzpicture}
   \pgfplotsset{ybar stacked, ymin=200.000, ymax=1500.000, width=\textwidth, enlargelimits=0.1,
   grid=both, grid style={line width=.2pt, draw=gray!20}, 
   xmin=0, xmax = 4,
 }
 
 \begin{axis}
 [
 ylabel={Total \# iterations}, 
 xticklabels={0.25, 0.5, 1, 2, 4},
 xlabel = {Surface tension - $\gamma$},
 ytick={200, 600, 1000, 1400}, 
 height=5cm,
 xtick=data,
 bar width=8pt,
 bar shift=-25pt, 
 height=5cm,
 scaled y ticks = false,
 ]  
 
 \pgfplotsinvokeforeach {1}{
     \addplot[color=neworange1, fill=neworange1] table [x expr=\coordindex, y index=#1] {\datatableGammaIterMonoSemi};}
     
\end{axis} 

 \begin{axis}
 [
 bar width=8pt,
 bar shift=-15pt, 
 xtick=\empty,
 ytick=\empty, 
 height=5cm,
 ]  
 
 \pgfplotsinvokeforeach {1}{
     \addplot[color=newblue1, fill=newblue1] table [x expr=\coordindex, y index=#1] {\datatableGammaIterMono};}
\end{axis} 

 \begin{axis}
 [
 bar width=8pt,
 bar shift=-5pt, 
 xtick=\empty,
 ytick=\empty, 
 height=5cm,
 ]  
 
 \pgfplotsinvokeforeach {1}{
     \addplot[color=newgreen1, fill=newgreen1] table [x expr=\coordindex, y index=#1] {\datatableGammaIterSplitSemi};}
\end{axis} 

 \begin{axis}
 [
 bar width=8pt,
 bar shift=5pt, 
 xtick=\empty,
 ytick=\empty, 
 height=5cm,
 ]  
 
 \pgfplotsinvokeforeach {1}{
     \addplot[color=newpurple1, fill=newpurple1] table [x expr=\coordindex, y index=#1] {\datatableGammaIterSplit};}
\end{axis} 

 \begin{axis}
 [
 bar width=8pt,
 bar shift=15pt, 
 xtick=\empty,
 ytick=\empty, 
 height=5cm,
 ]  
 
 \pgfplotsinvokeforeach {1}{
     \addplot[color=newred1, fill=newred1] table [x expr=\coordindex, y index=#1] {\datatableGammaIterFSSemi};}
     
\end{axis} 

 \begin{axis}
 [
 bar width=8pt,
 bar shift=25pt, 
 xtick=\empty,
 ytick=\empty, 
 height=5cm,
 ]  
 
 \pgfplotsinvokeforeach {1}{
     \addplot[color=neworange2, fill=neworange2] table [x expr=\coordindex, y index=#1] {\datatableGammaIterFS};}
     
\end{axis} 

 \begin{axis}
 [
 bar width=0pt,
 bar shift=10pt, 
 legend columns=2,
 legend style={at={(0.8,0.99)}, anchor=north}, 
 xtick=\empty,
 legend cell align=left,
 legend style={
     /tikz/column 2/.style={
         column sep=-1.3pt,
         row sep=-20pt 
     },
     /tikz/column 6/.style={ 
         column sep=-1.3pt,
         row sep=-20pt 
     }},
 ytick=\empty,
 height=5cm,
 ]  
 \pgfplotsinvokeforeach {1}{
     \addplot[color=neworange1, fill=neworange1] table [x expr=\coordindex, y index=#1] {\datatableEmpty};}
 \pgfplotsinvokeforeach {1}{
     \addplot[color=newblue1, fill=newblue1] table [x expr=\coordindex, y index=#1] {\datatableEmpty};}
 \pgfplotsinvokeforeach {1}{
     \addplot[color=newgreen1, fill=newgreen1] table [x expr=\coordindex, y index=#1] {\datatableEmpty};}
 \pgfplotsinvokeforeach {1}{
     \addplot[color=newpurple1, fill=newpurple1] table [x expr=\coordindex, y index=#1] {\datatableEmpty};}
 \pgfplotsinvokeforeach {1}{
     \addplot[color=newred1, fill=newred1] table [x expr=\coordindex, y index=#1] {\datatableEmpty};}
 \pgfplotsinvokeforeach {1}{
     \addplot[color=neworange2, fill=neworange2] table [x expr=\coordindex, y index=#1] {\datatableEmpty};}
     
 \legend{
 Semi-impl. Mono, 
 Mono, 
 Semi-impl. Split,
 Split,
 Semi-impl. FS,
 FS,
 }
 \end{axis}

 \end{tikzpicture}
 \caption{Total number of iterations for simulations with different surface tension parameters.}
 \end{subfigure}
 \begin{subfigure}{0.9\linewidth}
     \pgfplotstableread{
0 0
1 0
80 0
}\datatableEmpty

\pgfplotstableread{
0.25	140.6552622
0.5	125.0226722
1	106.5341723
2	93.66514444
4	80.19540572
}\datatableGammaTimeMonoSemi

\pgfplotstableread{
0.25	101.8442953
0.5		82.05285931
1		90.22534823
2		110.4022477
4		108.905005
}\datatableGammaTimeMono

\pgfplotstableread{
0.25	105.1253779
0.5		73.37491894
1		54.56369686
2		43.3175683
4	35.78557754
}\datatableGammaTimeSplitSemi

\pgfplotstableread{
0.25	152.7891393
0.5	119.4746678
1	92.82543612
2	81.59262156
4	58.12738442
}\datatableGammaTimeSplit

\pgfplotstableread{
0.25	122.1322029
0.5	83.54054856
1	60.85557985
2	49.22734928
4	39.25029492
}\datatableGammaTimeFSSemi

\pgfplotstableread{
0.25	164.4478703
0.5	133.1574991
1	100.1371615
2	90.08925366
4	65.02649927
}\datatableGammaTimeFS

  \begin{tikzpicture}
   \pgfplotsset{ybar stacked, ymin=15.000, ymax=200.000, width=\textwidth, enlargelimits=0.1,
   grid=both, grid style={line width=.2pt, draw=gray!20}, 
   xmin=0, xmax = 4,
 }

 \begin{axis}
 [
 ylabel={Total simation time}, 
 xticklabels={0.25, 0.5, 1, 2, 4},
 xlabel = {Surface tension - $\gamma$},
 ytick={20, 60, 100, 140, 180}, 
 height=5cm,
 xtick=data,
 bar width=8pt,
 bar shift=-25pt, 
 height=5cm,
 scaled y ticks = false,
 ]  
 
 \pgfplotsinvokeforeach {1}{
     \addplot[color=neworange1, fill=neworange1] table [x expr=\coordindex, y index=#1] {\datatableGammaTimeMonoSemi};}
     
\end{axis} 

 \begin{axis}
 [
 bar width=8pt,
 bar shift=-15pt, 
 xtick=\empty,
 ytick=\empty, 
 height=5cm,
 ]  
 
 \pgfplotsinvokeforeach {1}{
     \addplot[color=newblue1, fill=newblue1] table [x expr=\coordindex, y index=#1] {\datatableGammaTimeMono};}
\end{axis} 

 \begin{axis}
 [
 bar width=8pt,
 bar shift=-5pt, 
 xtick=\empty,
 ytick=\empty, 
 height=5cm,
 ]  
 
 \pgfplotsinvokeforeach {1}{
     \addplot[color=newgreen1, fill=newgreen1] table [x expr=\coordindex, y index=#1] {\datatableGammaTimeSplitSemi};}
\end{axis} 

 \begin{axis}
 [
 bar width=8pt,
 bar shift=5pt, 
 xtick=\empty,
 ytick=\empty, 
 height=5cm,
 ]  
 
 \pgfplotsinvokeforeach {1}{
     \addplot[color=newpurple1, fill=newpurple1] table [x expr=\coordindex, y index=#1] {\datatableGammaTimeSplit};}
\end{axis} 

 \begin{axis}
 [
 bar width=8pt,
 bar shift=15pt, 
 xtick=\empty,
 ytick=\empty, 
 height=5cm,
 ]  
 
 \pgfplotsinvokeforeach {1}{
     \addplot[color=newred1, fill=newred1] table [x expr=\coordindex, y index=#1] {\datatableGammaTimeFSSemi};}
     
\end{axis} 

 \begin{axis}
 [
 bar width=8pt,
 bar shift=25pt, 
 xtick=\empty,
 ytick=\empty, 
 height=5cm,
 ]  
 
 \pgfplotsinvokeforeach {1}{
     \addplot[color=neworange2, fill=neworange2] table [x expr=\coordindex, y index=#1] {\datatableGammaTimeFS};}
     
\end{axis} 

 \begin{axis}
 [
 bar width=0pt,
 bar shift=10pt, 
 legend columns=2,
 legend style={at={(0.8,0.98)}, anchor=north}, 
 xtick=\empty,
 legend cell align=left,
 legend style={
     /tikz/column 2/.style={
         column sep=-1.3pt,
         row sep=-20pt 
     },
     /tikz/column 6/.style={ 
         column sep=-1.3pt,
         row sep=-20pt 
     }},
 ytick=\empty,
 height=5cm,
 ]  
 \pgfplotsinvokeforeach {1}{
     \addplot[color=neworange1, fill=neworange1] table [x expr=\coordindex, y index=#1] {\datatableEmpty};}
 \pgfplotsinvokeforeach {1}{
     \addplot[color=newblue1, fill=newblue1] table [x expr=\coordindex, y index=#1] {\datatableEmpty};}
 \pgfplotsinvokeforeach {1}{
     \addplot[color=newgreen1, fill=newgreen1] table [x expr=\coordindex, y index=#1] {\datatableEmpty};}
 \pgfplotsinvokeforeach {1}{
     \addplot[color=newpurple1, fill=newpurple1] table [x expr=\coordindex, y index=#1] {\datatableEmpty};}
 \pgfplotsinvokeforeach {1}{
     \addplot[color=newred1, fill=newred1] table [x expr=\coordindex, y index=#1] {\datatableEmpty};}
 \pgfplotsinvokeforeach {1}{
     \addplot[color=neworange2, fill=neworange2] table [x expr=\coordindex, y index=#1] {\datatableEmpty};}
     
 \legend{
 Semi-impl. Mono, 
 Mono, 
 Semi-impl.  Split,
Split,
 Semi-impl. FS,
 FS,
 }
 \end{axis}

 \end{tikzpicture}

 \caption{Total computational time for simulations with different surface tension parameters.}
  \label{fig:gammatime}
 \end{subfigure}
 \caption{''Semi-impl. Mono'' refers to a monolithic Newton solver applied to the discrete system \eqref{eq:ch1split_sd}--\eqref{eq:pressuresplit_sd}, ''Semi-impl. Split'' refers to the scheme \eqref{eq:ch1split_alter_sd}--\eqref{eq:pressuresplit_alter_sd}, and ''Semi-impl. FS'' refers to a three-way decoupling solver, sequentially solving the Cahn-Hilliard, the elasticity and the flow subsystems, for the semi-implicit time discretization \eqref{eq:ch1split_sd}--\eqref{eq:pressuresplit_sd}. On the other hand, ''Mono'', ''Split'' and ''FS'', refer to the same solution strategies but applied to the discrete system arising from discretizing with the implicit Euler method and the Eyre convex-concave split \cite{Eyre1998} of the double-well potential.}
 \label{fig:gamma}
\end{figure}

\subsubsection{Dependence on swelling parameter $\xi$}
In Figure~\ref{fig:xi} we present an equivalent study, but with varying values of the swelling parameter $\xi$. The remaining material and simulation parameters are found in Table~\ref{tab:simulation_parameters}, specifically $\gamma = 1$. Here, we note that increasing the swelling parameter increases the computational challenge for the (non-)linear solvers. We expect that this is due to the same reasoning as for the surface tension parameter; namely that the coupling strength increases with increasing swelling parameter. Notice, in particular, that for the largest value of the swelling parameter, we failed to get convergence for the implicit monolithic Newton solver, which gives evidence of the robustness of both the semi-implicit time discretization and the decoupling approaches. In general, the results are similar to those for the study with varying surface tension parameters.

\pgfplotstableread{
0 0
1 0
80 0
}\datatableEmpty

\pgfplotstableread{
0.0625 202
0.125 202
0.25 202
0.5 264
}\datatableXiIterMonoSemi

\pgfplotstableread{
0.0625 203
0.125 203
0.25 284
0.5 0
}\datatableXiIterMono

\pgfplotstableread{
0.0625 211
0.125 301
0.25 401
0.5 571
}\datatableXiIterSplitSemi

\pgfplotstableread{
0.0625 257
0.125 332
0.25 680
0.5 994
}\datatableXiIterSplit

\pgfplotstableread{
0.0625 211
0.125 301
0.25 401
0.5 564
}\datatableXiIterFSSemi

\pgfplotstableread{
0.0625 254
0.125 330
0.25 677
0.5 985
}\datatableXiIterFS

\begin{figure}
\centering
\hspace{-0.79cm}
\begin{subfigure}{0.9\linewidth}
  \begin{tikzpicture}
   \pgfplotsset{ybar stacked, ymin=200.000, ymax=1100.000, width=\textwidth, enlargelimits=0.1,
   grid=both, grid style={line width=.2pt, draw=gray!20}, 
   xmin=0, xmax = 3,
 }
 
 \begin{axis}
 [
 ylabel={Total \# iterations}, 
 xticklabels={0.0625, 0.125, 0.25, 0.5},
 xlabel = {Swelling parameter - $\xi$},
 ytick={200, 600, 1000}, 
 height=5cm,
 xtick=data,
 bar width=8pt,
 bar shift=-25pt, 
 height=5cm,
 scaled y ticks = false,
 ]  
 
 \pgfplotsinvokeforeach {1}{
     \addplot[color=neworange1, fill=neworange1] table [x expr=\coordindex, y index=#1] {\datatableXiIterMonoSemi};}
     
\end{axis} 

 \begin{axis}
 [
 bar width=8pt,
 bar shift=-15pt, 
 xtick=\empty,
 ytick=\empty, 
 height=5cm,
 ]  
 
 \pgfplotsinvokeforeach {1}{
     \addplot[color=newblue1, fill=newblue1] table [x expr=\coordindex, y index=#1] {\datatableXiIterMono};}
\end{axis} 

 \begin{axis}
 [
 bar width=8pt,
 bar shift=-5pt, 
 xtick=\empty,
 ytick=\empty, 
 height=5cm,
 ]  
 
 \pgfplotsinvokeforeach {1}{
     \addplot[color=newgreen1, fill=newgreen1] table [x expr=\coordindex, y index=#1] {\datatableXiIterSplitSemi};}
\end{axis} 

 \begin{axis}
 [
 bar width=8pt,
 bar shift=5pt, 
 xtick=\empty,
 ytick=\empty, 
 height=5cm,
 ]  
 
 \pgfplotsinvokeforeach {1}{
     \addplot[color=newpurple1, fill=newpurple1] table [x expr=\coordindex, y index=#1] {\datatableXiIterSplit};}
\end{axis} 

 \begin{axis}
 [
 bar width=8pt,
 bar shift=15pt, 
 xtick=\empty,
 ytick=\empty, 
 height=5cm,
 ]  
 
 \pgfplotsinvokeforeach {1}{
     \addplot[color=newred1, fill=newred1] table [x expr=\coordindex, y index=#1] {\datatableXiIterFSSemi};}
     
\end{axis} 

 \begin{axis}
 [
 bar width=8pt,
 bar shift=25pt, 
 xtick=\empty,
 ytick=\empty, 
 height=5cm,
 ]  
 
 \pgfplotsinvokeforeach {1}{
     \addplot[color=neworange2, fill=neworange2] table [x expr=\coordindex, y index=#1] {\datatableXiIterFS};}
     
\end{axis} 

 \begin{axis}
 [
 bar width=0pt,
 bar shift=10pt, 
 legend columns=2,
 legend style={at={(0.2,0.95)}, anchor=north}, 
 xtick=\empty,
 legend cell align=left,
 legend style={
     /tikz/column 2/.style={
         column sep=-1.3pt,
         row sep=-20pt 
     },
     /tikz/column 6/.style={ 
         column sep=-1.3pt,
         row sep=-20pt 
     }},
 ytick=\empty,
 height=5cm,
 ]  
 \pgfplotsinvokeforeach {1}{
     \addplot[color=neworange1, fill=neworange1] table [x expr=\coordindex, y index=#1] {\datatableEmpty};}
 \pgfplotsinvokeforeach {1}{
     \addplot[color=newblue1, fill=newblue1] table [x expr=\coordindex, y index=#1] {\datatableEmpty};}
 \pgfplotsinvokeforeach {1}{
     \addplot[color=newgreen1, fill=newgreen1] table [x expr=\coordindex, y index=#1] {\datatableEmpty};}
 \pgfplotsinvokeforeach {1}{
     \addplot[color=newpurple1, fill=newpurple1] table [x expr=\coordindex, y index=#1] {\datatableEmpty};}
 \pgfplotsinvokeforeach {1}{
     \addplot[color=newred1, fill=newred1] table [x expr=\coordindex, y index=#1] {\datatableEmpty};}
 \pgfplotsinvokeforeach {1}{
     \addplot[color=neworange2, fill=neworange2] table [x expr=\coordindex, y index=#1] {\datatableEmpty};}
     
 \legend{
 Semi-impl. Mono, 
 Mono, 
 Semi-impl.  Split,
  Split,
 Semi-impl. FS,
 FS,
 }
 \end{axis}

 \end{tikzpicture}
 
 \caption{Total number of iterations for simulations with different swelling parameters.}
 \label{fig:xiiter}
 \end{subfigure}
 \begin{subfigure}{0.9\linewidth}
     \pgfplotstableread{
0 0
1 0
80 0
}\datatableEmpty

\pgfplotstableread{
0.0625 80.10626078
0.125 82.33554316
0.25 81.32613254
0.5 109.4604158
}\datatableXiTimeMonoSemi

\pgfplotstableread{
0.0625 81.35493159
0.125 82.04005098
0.25 113.3371298
0.5 0
}\datatableXiTimeMono

\pgfplotstableread{
0.0625 18.38070607
0.125 25.24320793
0.25 33.69879818
0.5 53.71665263
}\datatableXiTimeSplitSemi

\pgfplotstableread{
0.0625 20.01111579
0.125 27.31061411
0.25 55.76160026
0.5 88.40253735
}\datatableXiTimeSplit

\pgfplotstableread{
0.0625 20.82348657
0.125 27.03825688
0.25 38.61681962
0.5 60.45952535
}\datatableXiTimeFSSemi

\pgfplotstableread{
0.0625 21.01718736
0.125 27.59564543
0.25 62.48066998
0.5 99.8401587
}\datatableXiTimeFS

  \begin{tikzpicture}
   \pgfplotsset{ybar stacked, ymin=15.000, ymax=150.000, width=\textwidth, enlargelimits=0.1,
   grid=both, grid style={line width=.2pt, draw=gray!20}, 
   xmin=0, xmax = 3,
 }
 
 \begin{axis}
 [
 ylabel={Total simation time}, 
 xticklabels={0.0625, 0.125, 0.25, 0.5},
 xlabel = {Swelling parameter - $\xi$},
 ytick={20, 60, 100, 140}, 
 height=5cm,
 xtick=data,
 bar width=8pt,
 bar shift=-25pt, 
 height=5cm,
 scaled y ticks = false,
 ]  
 
 \pgfplotsinvokeforeach {1}{
     \addplot[color=neworange1, fill=neworange1] table [x expr=\coordindex, y index=#1] {\datatableXiTimeMonoSemi};}
     
\end{axis} 

 \begin{axis}
 [
 bar width=8pt,
 bar shift=-15pt, 
 xtick=\empty,
 ytick=\empty, 
 height=5cm,
 ]  
 
 \pgfplotsinvokeforeach {1}{
     \addplot[color=newblue1, fill=newblue1] table [x expr=\coordindex, y index=#1] {\datatableXiTimeMono};}
\end{axis} 

 \begin{axis}
 [
 bar width=8pt,
 bar shift=-5pt, 
 xtick=\empty,
 ytick=\empty, 
 height=5cm,
 ]  
 
 \pgfplotsinvokeforeach {1}{
     \addplot[color=newgreen1, fill=newgreen1] table [x expr=\coordindex, y index=#1] {\datatableXiTimeSplitSemi};}
\end{axis} 

 \begin{axis}
 [
 bar width=8pt,
 bar shift=5pt, 
 xtick=\empty,
 ytick=\empty, 
 height=5cm,
 ]  
 
 \pgfplotsinvokeforeach {1}{
     \addplot[color=newpurple1, fill=newpurple1] table [x expr=\coordindex, y index=#1] {\datatableXiTimeSplit};}
\end{axis} 

 \begin{axis}
 [
 bar width=8pt,
 bar shift=15pt, 
 xtick=\empty,
 ytick=\empty, 
 height=5cm,
 ]  
 
 \pgfplotsinvokeforeach {1}{
     \addplot[color=newred1, fill=newred1] table [x expr=\coordindex, y index=#1] {\datatableXiTimeFSSemi};}
     
\end{axis} 

 \begin{axis}
 [
 bar width=8pt,
 bar shift=25pt, 
 xtick=\empty,
 ytick=\empty, 
 height=5cm,
 ]  
 
 \pgfplotsinvokeforeach {1}{
     \addplot[color=neworange2, fill=neworange2] table [x expr=\coordindex, y index=#1] {\datatableXiTimeFS};}
     
\end{axis} 

 \begin{axis}
 [
 bar width=0pt,
 bar shift=10pt, 
 legend columns=2,
 legend style={at={(0.2,0.95)}, anchor=north}, 
 xtick=\empty,
 legend cell align=left,
 legend style={
     /tikz/column 2/.style={
         column sep=-1.3pt,
         row sep=-20pt 
     },
     /tikz/column 6/.style={ 
         column sep=-1.3pt,
         row sep=-20pt 
     }},
 ytick=\empty,
 height=5cm,
 ]  
 \pgfplotsinvokeforeach {1}{
     \addplot[color=neworange1, fill=neworange1] table [x expr=\coordindex, y index=#1] {\datatableEmpty};}
 \pgfplotsinvokeforeach {1}{
     \addplot[color=newblue1, fill=newblue1] table [x expr=\coordindex, y index=#1] {\datatableEmpty};}
 \pgfplotsinvokeforeach {1}{
     \addplot[color=newgreen1, fill=newgreen1] table [x expr=\coordindex, y index=#1] {\datatableEmpty};}
 \pgfplotsinvokeforeach {1}{
     \addplot[color=newpurple1, fill=newpurple1] table [x expr=\coordindex, y index=#1] {\datatableEmpty};}
 \pgfplotsinvokeforeach {1}{
     \addplot[color=newred1, fill=newred1] table [x expr=\coordindex, y index=#1] {\datatableEmpty};}
 \pgfplotsinvokeforeach {1}{
     \addplot[color=neworange2, fill=neworange2] table [x expr=\coordindex, y index=#1] {\datatableEmpty};}
     
 \legend{
 Semi-impl. Mono, 
 Mono, 
 Semi-impl.  Split,
  Split,
 Semi-impl. FS,
 FS,
 }
 \end{axis}

 \end{tikzpicture}
 
 \caption{Total computational time for simulations with different swelling parameters.}
 \label{fig:xitime}
 \end{subfigure}
 \caption{''Semi-impl. Mono'' refers to a monolithic Newton solver applied to the discrete system \eqref{eq:ch1split_sd}--\eqref{eq:pressuresplit_sd}, ''Semi-impl. Split'' refers to the scheme \eqref{eq:ch1split_alter_sd}--\eqref{eq:pressuresplit_alter_sd}, and ''Semi-impl. FS'' refers to a three-way decoupling solver, sequentially solving the Cahn-Hilliard, the elasticity and the flow subsystems, for the semi-implicit time discretization \eqref{eq:ch1split_sd}--\eqref{eq:pressuresplit_sd}. On the other hand, ''Mono'', ''Split'' and ''FS'', refer to the same solution strategies but applied to the discrete system arising from discretizing with the implicit Euler method and the Eyre convex-concave split \cite{Eyre1998} of the double-well potential.}
 \label{fig:xi}
\end{figure}
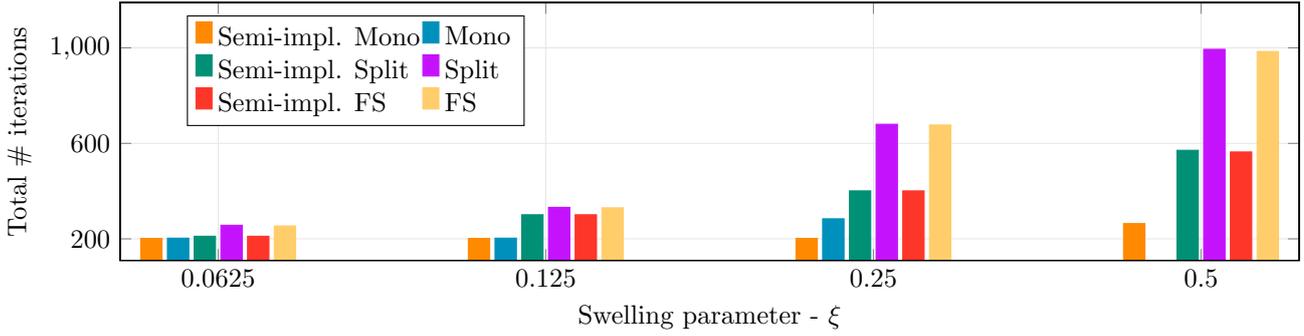
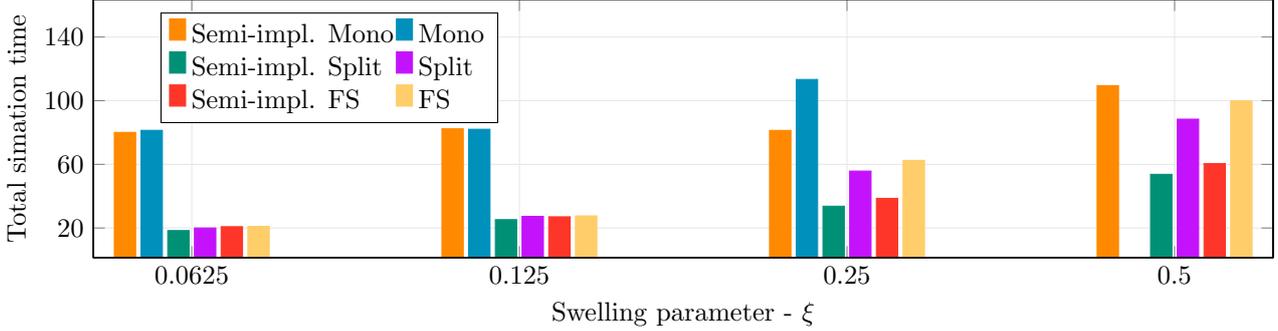

\subsubsection{Concerning the evolution of energy}
In Figure~\ref{fig:energy_plot} the evolution of the free energy is plotted over time for both of the previous experiments employing the semi-implicit time discretization \eqref{eq:model_disc_sd}. In general, we observe that the trend is a decreasing free energy. However, there are some small fluctuations happening for the cases with $\gamma = 0.25$ and $\xi = 0.5$, which correspond to the highest coupling strength respectively. 
This is in line with Remark~\ref{rem:disc_ener_diss_sd}. One means of remedy to obtain time-step control over the strain-related terms is to add regularization to the model by incorporating viscoelastic effects as done for the analysis in \cite{Fritz2024, Garcke2025, Riethmueller2025, Abels2025} and the numerical scheme in \cite{Brunk2025}. However, this relaxation alters the model formulation and it is beyond the scope of this paper to investigate it further.

\begin{figure}
    \centering
    \begin{subfigure}{0.45\linewidth}
        \includegraphics[width=\linewidth]{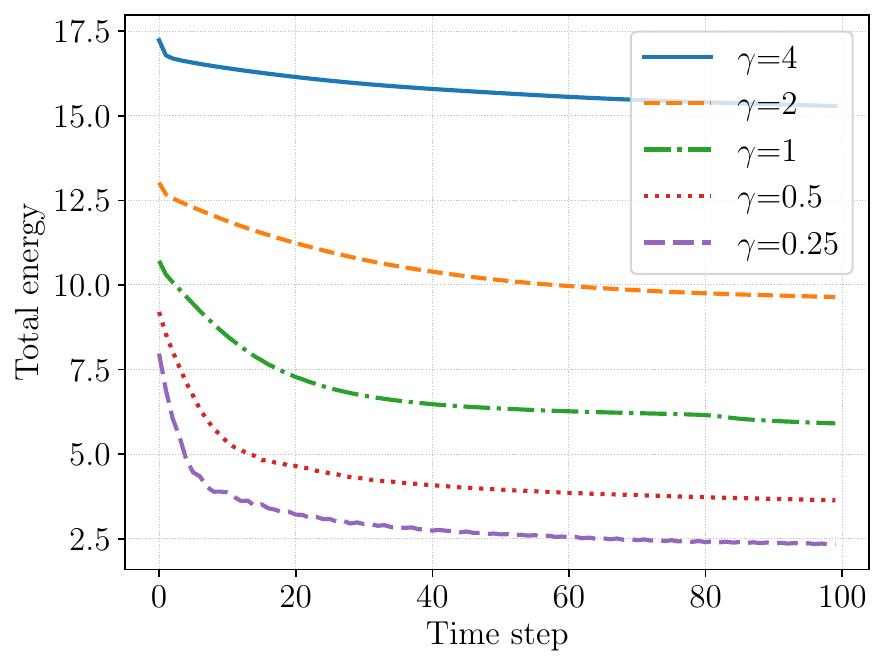}
    \end{subfigure}
    \begin{subfigure}{0.437\linewidth}
        \includegraphics[width = \linewidth]{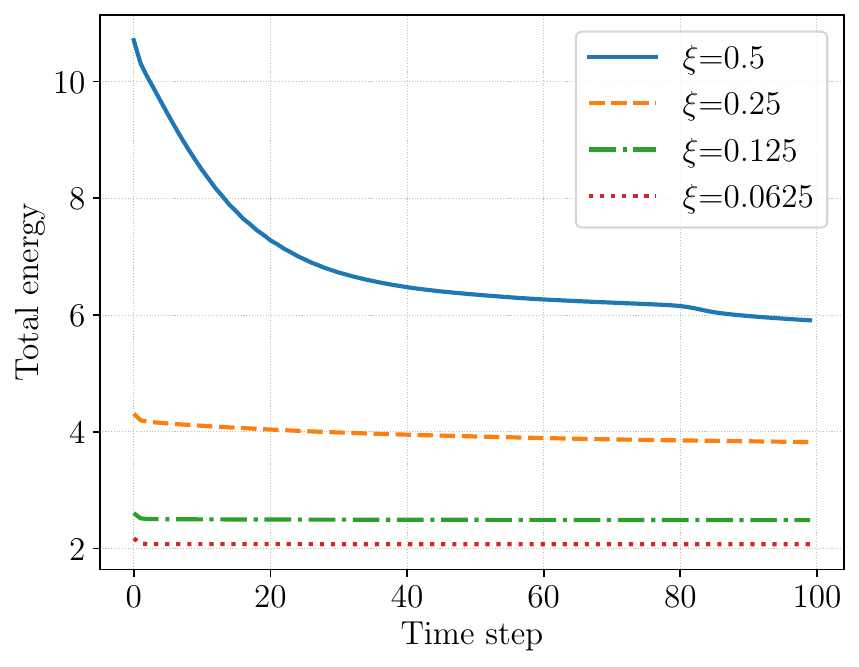}
    \end{subfigure}
    \caption{Total energy for each time step for varying surface tension (left) and swelling (right) parameters. Depicted are the energies calculated for the semi-implicit time discretization \eqref{eq:model_disc_sd}, using the alternating minimization method \eqref{eq:model_disc_alter_sd} to solve the systems in each time step.}
    \label{fig:energy_plot}
\end{figure}

\section{Conclusions} \label{sec:conclusions}
In this paper, we presented a new semi-implicit time discretization for the Cahn-Hilliard-Biot equations that is tailor-made to give nonlinear discrete systems suitable for general linearization methods. Moreover, we proved that alternating minimization, with respect to the Cahn-Hilliard and the Biot subsystems, converges when applied to solve these discrete systems, with quite general spatial discretizations. We round off the paper with a short numerical study of the introduced time discretization and iterative decoupling methods, showing the promise of the proposed solution strategy.

Possible further directions for research include the development of tailor-made preconditioned iterative solvers replacing the default linear solvers (from PETSc) used so far and the implementation of Anderson acceleration \cite{Anderson1965, Storvik2023} to accelerate the fixed-point-like iteration schemes. 

\medskip
\paragraph{Acknowledgement}
Funded in part by Deutsche Forschungsgemeinschaft (DFG, German Research Foundation) under Germany's Excellence Strategy - EXC 2075 – 390740016. CR acknowledges the support by the Stuttgart Center for Simulation Science (SC SimTech).

\appendix
\counterwithin{lemma}{section}
\renewcommand{\thelemma}{\arabic{lemma}\thesection}

\section{Abstract theory on the convergence of alternating minimization}
We refer to \cite{Both2022, Both2019} for abstract theory on the convergence of alternating minimization schemes and its application to thermo-poro-visco-elastic processes in porous media. The following lemma recapitulates the appropriate result from \cite{Both2022} in terms of the present article.

\begin{lemma}
Assume that there exist (semi-)norms $\lVert (\cdot,\cdot,\cdot) \rVert: \mathcal{V}_{h,0}^\mathrm{ch} \times \mathcal{V}_h^{\bm u} \times \mathcal{V}_{h,0}^{\mathrm{f}} \rightarrow \mathbb{R}_0^+$, $\lVert \cdot \rVert_\mathrm{ch}: \mathcal{V}_{h,0}^\mathrm{ch} \rightarrow \mathbb{R}_0^+$, and $\lVert (\cdot,\cdot) \rVert_\mathrm{b}: \mathcal{V}_h^{\bm u} \times \mathcal{V}_{h,0}^{\mathrm{f}} \rightarrow \mathbb{R}_0^+$, related by the inequalities
\begin{align}
 \lVert (\pf_h, \bm u_h, \theta_h) \rVert^2 \geq \beta_\mathrm{ch} \lVert \pf_h \rVert_\mathrm{ch}^2
 \label{eq:rel_norm_ch}
\end{align}
and
\begin{align}
 \lVert (\pf_h, \bm u_h, \theta_h) \rVert^2 \geq \beta_\mathrm{b} \lVert (\bm u_h, \theta_h) \rVert_\mathrm{b}^2
 \label{eq:rel_norm_b}
\end{align}
for all $\pf_h \in \mathcal{V}_{h,0}^\mathrm{ch}$, $\bm u_h \in \mathcal{V}_h^{\bm u}$ and $\theta_h \in \mathcal{V}_{h,0}^{\mathrm{f}}$, and some constants $\beta_\mathrm{ch}, \beta_\mathrm{b} \geq 0$. Furthermore, let the potential $\calH: \mathcal{\bar V}_h^{\mathrm{ch},n} \times \mathcal{V}_h^{\bm u} \times \mathcal{\bar V}_h^{\mathrm{f},n} \rightarrow \mathbb{R}$ be given. \\
If
\begin{itemize}
 \item $\calH$ is (quasi-)strongly convex with respect to the norm $\lVert (\cdot,\cdot,\cdot) \rVert$ with convexity constant $\sigma \geq 0$, i.e.,
 \begin{align}
  \Big(\delta_{(\pf,\bm u,\theta)}\calH\!\left(\pf_h^1, \bm u_h^1, \theta_h^1\right) - \delta_{(\pf,\bm u,\theta)}\calH\!\left(\pf_h^2, \bm u_h^2, \theta_h^2\right), \left(\pf_h^1 - \pf_h^2, \bm u_h^1 - \bm u_h^2, \theta_h^1 - \theta_h^2\right)\Big) \geq \sigma \big\lVert \left(\pf_h^1 - \pf_h^2, \bm u_h^1 - \bm u_h^2, \theta_h^1 - \theta_h^2\right) \big\rVert^2
  \label{eq:H_conv}
 \end{align}
 for all $\pf_h^1, \pf_h^2 \in \mathcal{\bar V}_h^{\mathrm{ch},n}$, $\bm u_h^1, \bm u_h^2 \in \mathcal{V}_h^{\bm u}$, and $\theta_h^1, \theta_h^2 \in \mathcal{\bar V}_h^{\mathrm{f},n}$,
\end{itemize}
and
\begin{itemize}
 \item the variational derivatives of $\calH$ are Lipschitz continuous in the norms $\lVert \cdot \rVert_\mathrm{ch}$ and $\lVert (\cdot,\cdot) \rVert_\mathrm{b}$ respectively, i.e., there exist constants $L_\mathrm{ch}, L_\mathrm{b} > 0$ such that
 \begin{align}
  \Big(\delta_\pf \calH\!\left(\pf_h^1, \bm u_h, \theta_h\right) - \delta_\pf \calH\!\left(\pf_h^2, \bm u_h, \theta_h\right), \pf_h^1 - \pf_h^2\Big) \leq L_\mathrm{ch} \big\lVert \pf_h^1 - \pf_h^2 \big\rVert_\mathrm{ch}^2
  \label{eq:Lip_cont_ch}
 \end{align}
 for all $\pf_h^1, \pf_h^2 \in \mathcal{\bar V}_h^{\mathrm{ch},n}$, $\bm u_h \in \mathcal{V}_h^{\bm u}$ and $\theta_h \in \mathcal{\bar V}_h^{\mathrm{f},n}$,
 and
 \begin{align}
  \Big(\delta_{(\bm u,\theta)}\calH\!\left(\pf_h, \bm u_h^1, \theta_h^1\right) - \delta_{(\bm u,\theta)}\calH\!\left(\pf_h, \bm u_h^2, \theta_h^2\right), \left(\bm u_h^1 - \bm u_h^2, \theta_h^1 - \theta_h^2\right)\Big) \leq L_\mathrm{b} \big\lVert \left(\bm u_h^1 - \bm u_h^2, \theta_h^1 - \theta_h^2\right) \big\rVert_b^2
  \label{eq:Lip_cont_b}
 \end{align}
 for all $\pf_h \in \mathcal{\bar V}_h^{\mathrm{ch},n}$, $\bm u_h^1, \bm u_h^2 \in \mathcal{V}_h^{\bm u}$, and $\theta_h^1, \theta_h^2 \in \mathcal{\bar V}_h^{\mathrm{f},n}$,
\end{itemize}
then the alternating minimization scheme converges in the sense that
\begin{align*}
 \calH\!\left(\pf_h^{n,i},\bm u_h^{n,i},\theta_h^{n,i}\right) - \calH\!\left(\pf_h^{n},\bm u_h^{n},\theta_h^{n}\right) \leq \left(1 - \frac{\sigma \beta_\mathrm{ch}}{L_\mathrm{ch}}\right) \left(1 - \frac{\sigma\beta_\mathrm{b}}{L_\mathrm{b}}\right) \Big( \calH\!\left(\pf_h^{n,i-1},\bm u_h^{n,i-1},\theta_h^{n,i-1}\right) - \calH\!\left(\pf_h^{n},\bm u_h^{n},\theta_h^{n}\right) \Big),
\end{align*}
where $\left(\pf_h^n,\bm u_h^n,\theta_h^n\right) \in \mathcal{\bar V}_h^{\mathrm{ch},n} \times \mathcal{V}_h^{\bm u} \times \mathcal{\bar V}_h^{\mathrm{f},n}$ is the minimizer of $\calH$.
\label{lem:aux_conv_result}
\end{lemma}

\bibliographystyle{unsrt}
\bibliography{bibliography}

\end{document}